\documentclass[11pt]{amsart}
\usepackage{fullpage,url,mathrsfs,amssymb,color,multicol,enumerate}
\usepackage[all]{xy} 
\usepackage{verbatim,moreverb}

\usepackage[OT2,T1]{fontenc}

\definecolor{backgroundcolor}{rgb}{0.9,0.8,0.8}

\numberwithin{equation}{section}

\usepackage{color}
\newcommand{\vanish}[1]{}

\def\bbar#1{\setbox0=\hbox{$#1$}\dimen0=.2\ht0 \kern\dimen0 }
\newcommand{\defi}[1]{\textsf{#1}} 

\newenvironment{romanenum}{\hfill \begin{enumerate} }{\end{enumerate}}
\newenvironment{alphenum}{\hfill \begin{enumerate} }{\end{enumerate}}

\DeclareSymbolFont{cyrletters}{OT2}{wncyr}{m}{n}
\DeclareMathSymbol{\Sha}{\mathalpha}{cyrletters}{"58}
\newcommand\legendre[2]{\Bigl(\frac{#1}{#2}\Bigr) }   

\newcommand{\Aff}{{\mathbb A}}

  \newcommand{\FF}{{\mathbb F}}

\newcommand{\PP}{{\mathbb P}}
 \newcommand{\QQ}{{\mathbb Q}}
\newcommand{\RR}{{\mathbb R}}
\newcommand{\ZZ}{{\mathbb Z}}

\def\bbar#1{\setbox0=\hbox{$#1$}\dimen0=.2\ht0 \kern\dimen0 \overline{\kern-\dimen0 #1}}
\newcommand{\Qbar}{{\overline{\mathbb Q}}}

 \newcommand{\FFbar}{\overline{\FF}}

\newcommand{\xx}{{\defi x}}
\newcommand{\yy}{{\defi y}}

\newcommand{\calE}{{\mathcal E}} 
\newcommand{\calF}{{\mathcal F}}
\newcommand{\calG}{{\mathcal G}}

\newcommand{\calL}{{\mathcal L}}

\newcommand{\calN}{{\mathcal N}}

\newcommand{\calR}{{\mathcal R}}
\newcommand{\calS}{{\mathcal S}}
\newcommand{\calT}{{\mathcal T}}

\newcommand{\calX}{{\mathcal X}}

\newcommand{\OO}{{\mathcal O}}

\DeclareMathOperator{\Drop}{drop}

\DeclareMathOperator{\Frob}{Frob}

\DeclareMathOperator{\Aut}{Aut}
\DeclareMathOperator{\Gal}{Gal}

\DeclareMathOperator{\Spec}{Spec}

\DeclareMathOperator{\MC}{MC}

\newcommand{\tors}{{\operatorname{tors}}}

\newcommand{\GL}{\operatorname{GL}}
\newcommand{\SL}{\operatorname{SL}}

\newcommand{\PSL}{\operatorname{PSL}}
\newcommand{\PSp}{\operatorname{PSp}}
\newcommand{\PSU}{\operatorname{PSU}}

\newcommand{\SO}{\operatorname{SO}}

\newcommand{\Or}{\operatorname{O}}
\newcommand{\POmega}{\operatorname{P}\!\Omega}

\DeclareMathOperator{\spin}{sp}
\DeclareMathOperator{\disc}{disc}

\newcommand{\ang}[2]{\langle #1,#2\rangle}

\def\CC{\mathbb C}

\DeclareMathOperator{\I}{\operatorname{I}}
\DeclareMathOperator{\II}{\operatorname{II}}
\DeclareMathOperator{\III}{\operatorname{III}}
\DeclareMathOperator{\IV}{\operatorname{IV}}

\DeclareMathOperator{\Kod}{\operatorname{Kod}}

\newtheorem{theorem}{Theorem}[section]
\newtheorem{lemma}[theorem]{Lemma}
\newtheorem{corollary}[theorem]{Corollary}
\newtheorem{proposition}[theorem]{Proposition}

\theoremstyle{definition}

\newtheorem{conjecture}[theorem]{Conjecture}

\theoremstyle{remark}
\newtheorem{remark}[theorem]{Remark}

\definecolor{webcolor}{rgb}{0,0,1}
\definecolor{webbrown}{rgb}{.6,0,0}
\usepackage[
        colorlinks,
        linkcolor=webbrown,  filecolor=webbrown,  citecolor=webbrown, 
        backref,
        pdfauthor={David Zywina}, 
        pdftitle={The inverse Galois problem for orthogonal groups},
]{hyperref}
\usepackage[alphabetic,backrefs,lite]{amsrefs} 

\begin{document}

\title[]{The inverse Galois problem for orthogonal groups}


\author{David Zywina}
\address{Department of Mathematics, Cornell University, Ithaca, NY 14853, USA}
\email{zywina@math.cornell.edu}
\urladdr{http://www.math.cornell.edu/~zywina/}

\begin{abstract}
We prove many new cases of the Inverse Galois Problem for those simple groups arising from orthogonal groups over finite fields.  For example, we show that the finite simple groups $\Omega_{2n+1}(p)$ and $\POmega_{4n}^+(p)$ both occur as the Galois group of a Galois extension of the rationals for all integers $n\geq 2$ and all primes $p\geq 5$.  We obtain our representations by studying families of twists of elliptic curves and using some known cases of the Birch and Swinnerton-Dyer conjecture along with a big monodromy result of Hall.
\end{abstract}

\thanks{This material is based upon work supported by the National Science Foundation under agreement No. DMS-1128155.}

\maketitle

\section{Introduction}

The \defi{Inverse Galois Problem} asks whether every finite group $G$ is isomorphic to the Galois group of some Galois extension of $\QQ$.  This problem is extremely difficult, even in the special case of non-abelian simple groups which we now restrict our attention to.  Many special cases are known, including alternating groups and all but one of the sporadic simple groups.  Several families of simple groups of Lie type are known to occur as Galois groups of an extension of $\QQ$, but usually with congruences imposed on the cardinality of the fields.   See \cite{MR1711577} for background and many examples.   Moreover, one can ask whether there is a \defi{regular} Galois extension $K/\QQ(t)$ with Galois group isomorphic to $G$ (by regular, we mean that $\QQ$ is algebraically closed in $K$).   If such an extension $K/\QQ(t)$ exists, then Hilbert's irreducibility theorem implies that there are infinitely many Galois extensions of $\QQ$ with Galois group isomorphic to $G$, and likewise over every number field.

\subsection{The groups}

We first introduce the simple groups that we are interested in realizing as Galois groups; further background can be found in \cite{MR2562037}*{\S3.7} and \cite{Atlas}*{\S2.4}.	  

Take any prime $\ell\geq 5$.  An \defi{orthogonal space} over $\FF_\ell$ is a finite dimensional $\FF_\ell$-vector space $V$ equipped with a non-degenerate and symmetric bilinear pairing $\ang{\,}{\:}\colon V\times V \to \FF_\ell$.     A homomorphism of orthogonal spaces is an $\FF_\ell$-linear map which is compatible with the respective pairings.   The \defi{orthogonal group} of $V$, denoted by $\Or(V)$, is the group of automorphisms of $V$ as an orthogonal space.    The \defi{special orthogonal group} $\SO(V)$ is the index $2$ subgroup of $\Or(V)$ consisting of those elements with determinant $1$.     For each $v\in V$ with $\ang{v}{v}\neq 0$, we have a reflection $r_v \in \Or(V)$ defined by $x\mapsto x - 2 \ang{x}{v}/\ang{v}{v} \cdot v$.   The \defi{spinor norm} of $V$ is the homomorphism 
\[
\spin\colon \Or(V)\to \FF_\ell^\times/(\FF_\ell^\times)^2
\]
characterized by the property that it satisfies $\spin(r_v) = \ang{v}{v}\cdot (\FF_\ell^\times)^2$ for every $v\in V$ with $\ang{v}{v}\neq 0$. The \defi{discriminant} of $V$ is $\disc(V):=\spin(-I)$; it can also be defined as the coset in $\FF_\ell^\times/(\FF_\ell^\times)^2$ represented by $\det(\ang{e_i}{e_j})$ where $\{e_1,\ldots, e_n\}$ is any basis of $V$.

Now fix an integer $n\geq 5$, and take any orthogonal space $V$ over $\FF_\ell$ of dimension $n$.    We define $\Omega(V)$ to be the subgroup of $\SO(V)$  consisting of elements with trivial spinor norm.    Using that $\FF_\ell^\times/(\FF_\ell^\times)^2$ has order $2$, one can show that $\Omega(V)$ is an index $2$ subgroup of $\SO(V)$.   The group $\Omega(V)$ is perfect and its center is either $\{I\}$ or $\{\pm I\}$.     Denote by $\POmega(V)$ the quotient of $\Omega(V)$ by its center.    
   
Suppose that $n$ is odd.    The group $\Omega(V)$ has trivial center and is simple.  The isomorphism class of $\Or(V)$, and hence also $\Omega(V)$, depends only on $n$ and $\ell$.  Denote by $\Omega_n(\ell)$ an abstract group isomorphic to $\Omega(V)$ (other common notation for this group is $O_n(\ell)$ and $B_{(n-1)/2}(\ell)$).

 Suppose that $n$ is even.  Up to isomorphism, there are two orthogonal spaces of dimension $n$ over $\FF_\ell$ and they are distinguishable by their discriminants.  We say that $V$ is \defi{split} if $\disc(V) = (-1)^{n/2}(\FF_\ell^\times)^2$ and \defi{non-split} otherwise (note that $V$ is split if and only if it is an orthogonal sum of hyperbolic planes).  Denote by $\POmega^+_n(\ell)$ and $\POmega^-_n(\ell)$ an abstract group isomorphic to $\POmega(V)$ when $V$ is split or non-split, respectively.  The groups $\POmega_n^+(\ell)$ and $\POmega_n^-(\ell)$ are both simple.  (Other common notation for $\POmega_n^+(\ell)$ is $O_n^+(\ell)$ and $D_{n/2}(\ell)$.  Other common notation for $\POmega_n^-(\ell)$ is $O_n^-(\ell)$ and ${}^2\! D_{n/2}(\ell^2)$.)

\subsection{Main results}

\begin{theorem} \label{T:Omega total}
Take any integer $n\geq 5$ and prime $\ell\geq 5$.
\begin{romanenum}
\item \label{T:Omega total i}
If $n$ is odd, then the simple group $\Omega_{n}(\ell)$ occurs as the Galois group of a regular extension of $\QQ(t)$.
\item \label{T:Omega total ii}
If $n \equiv 0 \pmod{4}$ or $\ell \equiv 1 \pmod{4}$, then the simple group $\POmega^+_{n}(\ell)$ occurs as the Galois group of a regular extension of $\QQ(t)$.
\item \label{T:Omega total iii}
If $n \equiv 2 \pmod{4}$ and $\ell\equiv 3 \pmod{4}$, then the simple group $\POmega^-_{n}(\ell)$ occurs as the Galois group of a regular extension of $\QQ(t)$.
\item \label{T:Omega total iv}
If $n$ is even and $2$, $3$, $5$ or $7$ is not a square modulo $\ell$, then the groups $\POmega^{+}_n(\ell)$ and $\POmega^{-}_n(\ell)$ both occur as the Galois group of a regular extension of $\QQ(t)$.
\end{romanenum}
\end{theorem}

The following is a restatement of Theorem~\ref{T:Omega total}(\ref{T:Omega total ii}) and (\ref{T:Omega total iii}).

\begin{corollary}  \label{C:Omega even}
Take any even integer $n\geq 6$ and prime $\ell \geq 5$.  If $V$ is an orthogonal space of dimension $n$ over $\FF_\ell$ with $\disc(V)=(\FF_\ell^\times)^2$, then the simple group $\POmega(V)$ occurs as the Galois group of a regular extension of $\QQ(t)$.
\end{corollary}

The following is a consequence of Theorem~\ref{T:Omega total} and the exceptional isomorphisms $\Omega_5(\ell) \cong \PSp_4(\FF_\ell)$, $\POmega^+_6(\ell)\cong \PSL_4(\FF_\ell)$ and $\POmega^-_6(\ell)\cong \PSU_4(\FF_\ell)$.  

\begin{corollary} Take any prime $\ell\geq 5$.
\begin{romanenum}
\item
The simple group $\PSp_4(\FF_\ell)$ occurs as the Galois group of a regular extension of $\QQ(t)$.  
\item
If $\ell$ is not congruent to $311$, $479$, $551$, $671$, $719$ and $839$ modulo $840$, then the simple group $\PSL_4(\FF_\ell)$ occurs as the Galois group of a regular extension of $\QQ(t)$.  
\item
If $\ell$ is not congruent to $1$, $121$, $169$, $289$, $361$ and $529$ modulo $840$, then the simple group $\PSU_4(\FF_\ell)$ occurs as the Galois group of a regular extension of $\QQ(t)$.
\end{romanenum}
\end{corollary}

\subsection{Some previous work and related cases} \label{SS:previously}

Reiter \cite{MR1695795} proved Theorem~\ref{T:Omega total}(\ref{T:Omega total i}) in the special case where $2$ or $3$ is not a square modulo $\ell$; in particular, it covers the case $\ell=3$ which we excluded.   Additional special cases of Theorem~\ref{T:Omega total}(\ref{T:Omega total i}) for $n=5$ and $7$ were proved by H\"afner \cite{MR1157316}.   Theorem~\ref{T:Omega total}(\ref{T:Omega total iv}) covers the various cases of the regular inverse Galois problem for $\POmega^{+}_n(\ell)$ and $\POmega^{-}_n(\ell)$ with $\ell\geq 5$ that are due to Reiter \cite{MR1695795} and Malle-Matzat \cite{MR1711577}*{\S10.2}.  The cases of the regular inverse Galois problem in Theorem~\ref{T:Omega total}(\ref{T:Omega total ii}) and (\ref{T:Omega total iii}) appear to be new.  \\

We now briefly discuss the excluded cases $n=3$ and $4$; we do not obtain non-abelian simple groups when $n\leq 2$.  These cases are especially interesting because of the exceptional isomorphisms
\[
\Omega_3(\ell) \cong \PSL_2(\FF_\ell),\quad \POmega_4^+(\ell) \cong \PSL_2(\FF_\ell)\times \PSL_2(\FF_\ell)\quad \text{and} \quad \POmega_4^-(\ell)\cong \PSL_2(\FF_{\ell^2}).
\]
The simple group $\PSL_2(\FF_\ell)$ is known to occur as the Galois group of a regular extension of $\QQ(t)$ if $2$, $3$, $5$ or $7$ is not a square modulo $\ell$; the cases $2$, $3$ and $7$ are due to Shih \cite{MR0332725} and $5$ is then due to Malle \cite{MR1199685}.  The conclusion of Theorem~\ref{T:Omega total}(\ref{T:Omega total i}) and Theorem~\ref{T:Omega total}(\ref{T:Omega total ii}) with $n=3$ and $n=4$, respectively, remains open.  

In \cite{Zywina-PSL2}, the author showed that $\PSL_2(\FF_\ell)$ occurs as a Galois group of an extension of $\QQ$ for all primes $\ell$.   The construction of such extensions in \cite{Zywina-PSL2} is similar to those of this paper; however, regular extensions of $\QQ(t)$ are not obtained.

The group $\PSL_2(\FF_{\ell^2})$ is already known to occur as the Galois group of a regular extension of $\QQ(t)$ if $2$, $3$, $5$ or $7$ is not a square modulo $\ell$;  see \cite{MR2059760} and \cite{MR1979499} for $2$ and $3$, \cite{MR972820} for $5$, and \cite{MR2282374} for $7$.   For many other $\ell$, the group $\PSL_2(\FF_{\ell^2})$ is known to occur as the Galois group of an extension of $\QQ$; for examples, see \cite{MR0419358}*{\S7}, \cite{MR1352266} and \cite{MR1800679}.\newline

The simple group $G_2(\ell)$ occurs as the Galois group of a regular extension of $\QQ(t)$ for all primes $\ell \geq 5$ (cf.~\cite{MR817266} for $\ell>5$ and \cite{MR817265} for $\ell=5$).   The simple group $E_8(\ell)$ occurs as the Galois group of a regular extension of $\QQ(t)$ for all primes $\ell \geq 7$, cf.~\cite{guralnick2012rational};  this was first shown to be true by Yun for all $\ell$ sufficiently large \cite{Yun:2013}.  Theorem~\ref{T:Omega total}(\ref{T:Omega total i}) and Theorem~\ref{T:Omega total}(\ref{T:Omega total ii}) with $n\equiv 0 \pmod{4}$ are the first cases where one has analogous results for finite simple groups of a fixed \emph{classical} Lie type.

\subsection{A special case}  

We now give an overview of the ideas behind the proof of Theorem~\ref{T:Omega total} in the special case $n=5$. In particular, for a fixed $\ell\geq 5$, we will describe a regular extension of $\QQ(t)$ with Galois group isomorphic to $\Omega_5(\ell)$. This section can be safely skipped and will not be referred to later on.

Define $S:=\{2,3,\ell\}$ and the ring $R:=\ZZ[S^{-1}]$.  Define the $R$-scheme 
\[
M=\Spec R[u, u^{-1}, (u-1)^{-1}, (u+1)^{-1}];
\]  
it is an open subscheme of $\Aff^1_R=\Spec R[u]$.

Let $k$ be any finite field that is an $R$-algebra, i.e., a finite field whose characteristic is not $2$, $3$ or $\ell$.   Denote the cardinality of $k$ by $q$.  Take any $m\in M(k)$, i.e., any $m\in k-\{0,1,-1\}$.  Let $E_m$ be the elliptic curve over the function field $k(t)$ defined by the Weierstrass equation
\[
(t-m)\cdot \yy^2 = \xx^3+ 3(t^2-1)^3  \xx -2(t^2-1)^5.
\]
Denote by $L(T,E_m)$ the \defi{$L$-function} of the elliptic curve $E_m/k(t)$, see \S\ref{SS:Lfunction} for details.  One can show that $L(T,E_m)$ is a polynomial in $\ZZ[T]$ of degree $5$.   For example, with $k=\FF_5$ and $m=1$, one can compute that $L(T,E_m)=  1- 2 T + T^2 - 5 T^3+ 2\cdot 5^3 T^4- 5^5T^5$.

Using the cohomological description of $L$-functions, we will construct an orthogonal space $V_\ell$ over $\FF_\ell$ of dimension $5$ and a continuous representation
\[
\theta_\ell\colon \pi_1(M) \to \Or(V_\ell)
\]
such that for any $k$ and $m\in M(k)$ as above, we have
\begin{equation} \label{E:compatibility 5 intro}
\det(I - \theta_\ell(\Frob_m) T) \equiv L(T/q, E_m) \pmod{\ell}.
\end{equation}
Here $\pi_1(M)$ is the \'etale fundamental group of $M$ (with suppressed base point) and $\Frob_m$ is the geometric Frobenius conjugacy class of $m$ in $\pi_1(M)$.  

The representation $\theta_\ell$ has \defi{big monodromy}, i.e.,  $\theta_\ell(\pi_1(M_\Qbar)) \supseteq \Omega(V_\ell)$.   This will be shown  following the approach of Hall in \cite{MR2372151} (we will directly use Hall's results for the cases with $n>5$).    The key step is to show that the group $\calR$ generated by the reflections in $\theta_\ell(\pi_1(M_{\Qbar}))$ acts irreducibly on $V_\ell$.   The classification of finite irreducible linear groups  generated by reflections then gives a finite number of small possibilities for $\calR$ that need to be ruled out to ensure that $\calR\supseteq \Omega(V_\ell)$.    

The image of $\theta_\ell$ can sometimes be the full orthogonal group $\Or(V_\ell)$ (in fact, this happens if $\ell \equiv \pm 3 \pmod{8}$).

Let $W$ be the open subscheme $\Spec R[u, u^{-1}, (u^2-3)^{-1}, (u^2+3)^{-1}]$ of $\Aff^1_R$.    The morphism $h\colon W\to M$ given by $w\mapsto (-w^2 + 3)/(w^2 + 3)$ is  \'etale of degree $2$, so we have a representation
\[
\vartheta_\ell\colon \pi_1(W) \xrightarrow{h_*} \pi_1(M) \xrightarrow{\theta_\ell} \Or(V_\ell).
\]
We claim that there are inclusions 
\begin{equation} \label{E:intro inclusions}
\Omega(V_\ell) \subseteq \vartheta_\ell(\pi_1(W_{\Qbar})) \subseteq \vartheta_\ell(\pi_1(W)) \subseteq \pm \Omega(V_\ell),
\end{equation}
where $\pm \Omega(V_\ell)$ is the group generated by $\Omega(V_\ell)$ and $-I$.  The natural map $\Omega(V_\ell)\to (\pm\Omega(V_\ell))/\{\pm I\}$ is an isomorphism, so the claimed inclusions give a surjective homomorphism
\[
\beta\colon \pi_1(W) \xrightarrow{\vartheta_\ell} \pm \Omega(V_\ell) \to (\pm \Omega(V_\ell))/\{\pm I\} \cong \Omega(V_\ell)
\]
that satisfies $\beta(\pi_1(W_{\Qbar})) = \Omega(V_\ell)$.  Therefore, $\beta$ gives rise to a regular extension of $\QQ(u)$, i.e., the function field of $W$, that is Galois with Galois group isomorphic to $\Omega(V_\ell) \cong \Omega_5(\ell)$; this gives the desired extension for the $n=5$ case of Theorem~\ref{T:Omega total}(\ref{T:Omega total i}).   We now explain why the inclusions of (\ref{E:intro inclusions}) hold.\\

We have the inclusion $\vartheta_\ell(\pi_1(W_\Qbar)) \supseteq \Omega(V_\ell)$ of (\ref{E:intro inclusions}) since the simple non-abelian group $\Omega(V_\ell)$ is a normal subgroup of $\theta_\ell(M_{\Qbar})$ and the cover $h$ is abelian.\\

Now let $\kappa$ be any coset of $\Omega(V_\ell)$ in $\theta_\ell(\pi_1(M))$ with $\det(\kappa)=\{ -1 \}$.   Since $\det(\kappa)=\{-1\}$, one can show that there is an element $A\in \kappa$ such that $\det(I-A)\neq 0$.    From a formula of Zassenhaus, we find that $\spin(-A)=2\det(I-A) \cdot (\FF_\ell^\times)^2$.
Using equidistribution, there is a prime $p\nmid 6\ell$ and an $m\in M(\FF_p)$ such that $A$ is conjugate to $\theta_\ell(\Frob_m)$ in $\Or(V_\ell)$.    We have $L(T/p,E_m) \equiv \det(I-AT) \pmod{\ell}$ and hence $L(1/p, E_m) \equiv \det(I-A) \pmod{\ell}$.  Therefore,  
\[
\spin(-A)= 2\cdot L(1/p,E_m) \cdot (\FF_\ell^\times)^2.
\]

The special value $L(1/p, E_m)$ is linked to the arithmetic of the curve $E_m/\FF_p(t)$.   We have $L(1/p,E_m)\neq 0$ (since it is non-zero modulo $\ell$), so the \defi{Birch and Swinnerton-Dyer (BSD) conjecture} predicts that the Mordell-Weil group $E_m(\FF_p(t))$ is finite.    In fact, this is known unconditionally by work of Artin and Tate.   Moreover, from Artin, Tate and Milne, the following refined version of BSD is known to hold: we have
\[
L(1/p,E_m) = \frac{|\Sha_{E_m}| \cdot c_{E_m} }{|E_m(\FF_p(t))|^2\cdot p^{-1+\chi_{E_m}}},
\]
where $c_{E_m}$ is the product of Tamagawa numbers of $E_m$ over the places of $\FF_p(t)$, $12\chi_{E_m}$ is the degree of the minimal discriminant of $E_m$, and $\Sha_{E_m}$ is the (finite!) Tate-Shafarevich group of $E_{m}$.   Since $\Sha_{E_m}$ is finite, a pairing of Cassels on $\Sha_{E_m}$ shows that $|\Sha_{E_m}|$ is a square.      An application of Tate's algorithm shows that $\chi_{E_m}=3$ and that $c_{E_m}$ is a power of $2$.   So $2L(1/p,E_m) \in 2c_{E_m} (\QQ^\times)^2$ and  $\spin(-A)= 2c_{E_m} \cdot (\FF_\ell^\times)^2$. Using Tate's algorithm, one can show that
$2c_{E_m}\in \{16,64\}$ if $-3(m^2-1) \in \FF_p$ is a square and $2c_{E_m}=32$ otherwise.  Therefore,
\[
\spin(-A) = \begin{cases}
      (\FF_\ell^\times)^2 & \text{ if $-3(m^2-1) \in \FF_p$ is a square,} \\
      2(\FF_\ell^\times)^2 & \text{ if $-3(m^2-1) \in \FF_p$ is not a square.}
\end{cases}
\]

Now suppose that $\kappa$ is actually a coset of $\Omega(V_\ell)$ in $\vartheta_\ell(\pi_1(W))$.  Then $m=h(w)$ for some $w\in W(\FF_p)$.   We have $-3(m^2-1)=6^2 w^2/(w^2+3)^2$ which is clearly a square (our degree $2$ cover $h\colon W\to M$ was chosen to ensure this held),  so $\spin(-A) = (\FF_\ell^\times)^2$.   We have $-A \in \Omega(V_\ell)$ since $\spin(-A) = (\FF_\ell^\times)^2$ and $\det(-A)=(-1)^5 \det(A) = 1$.  Therefore, $\kappa = A \Omega(V_\ell)=-\Omega(V_\ell)$.  
\\

We now know that $\vartheta_\ell(\pi_1(W))$ contains $\Omega(V_\ell)$ and the only possibly $\Omega(V_\ell)$-coset with determinant $-1$ is $-\Omega(V_\ell)$.   Therefore, $\vartheta_\ell(\pi_1(W))$ is either $\pm \Omega(V_\ell)$ or is a subgroup of $\SO(V_\ell)$.   So to explain the last inclusion of (\ref{E:intro inclusions}), we need only show that $\vartheta_\ell(\pi_1(W))$ contains an element with determinant $-1$.

For any $p\notin S$ and $m\in M(\FF_p)$, there is a functional equation 
\[
T^5 L(T^{-1}/p, E_m) = \varepsilon_{E_m} L(T/p,E_m),
\]
where $\varepsilon_{E_m} \in \{\pm 1\}$ is the \emph{root number} of $E_m$.   Since $A:=\theta_\ell(\Frob_m)$ belongs to $\Or(V_\ell)$, we have $T^5 \det(I-T^{-1}A) \equiv \det(-A) \det(I-TA)$.  From (\ref{E:compatibility 5 intro}), we deduce that $\det(-A)=\varepsilon_{E_m}$ and hence $\det(\theta_\ell(\Frob_m))=-\varepsilon_{E_m}$.   One can express $\varepsilon_{E_m}$ as a product of local root numbers and a computation shows that $-\varepsilon_{E_m}$ is $1$ if $-3m \in \FF_p$ is a square and $-1$ otherwise.

So if $\vartheta_\ell(\pi_1(W))$ is a subgroup of $\SO(V_\ell)$, then $-3 h(w) = -3(-w^2+3)/(w^2+3)$ is a square in $\FF_p$ for all primes $p\notin S$ and $w\in W(\FF_p)$.   This is easily seen to be false, so we deduce that $\vartheta_\ell(\pi_1(W))$ is not a subgroup of $\SO(V_\ell)$.

\subsection{Overview}
In \S\ref{S:elliptic background}, we give background on elliptic curves defined over global function fields.  We will mainly be interested in non-isotrivial elliptic curves $E$ defined over a function field $k(t)$, where $k$ is a finite field of order $q$ with $(q,6)=1$.   We will recall the definition of the \emph{$L$-function} $L(T,E)$ of $E$; it is a polynomial with integer coefficients.  For almost all primes $\ell$, we will construct an orthogonal space $V_{E,\ell}$ over $\FF_\ell$ and a representation 
\[
\theta_{E,\ell}\colon \Gal(\bbar{k}/k)\to  \Or(V_{E,\ell})
\] 
such that $\det(I - \theta_{E,\ell}(\Frob_q) T) \equiv L(T/q,E) \pmod{\ell}$, where $\Frob_q$ is the geometric Frobenius (i.e., the inverse of $x\mapsto x^q$).   The determinant of $\theta_{E,\ell}(\Frob_q)$ is related to the \emph{root number} of $E$.   In many case, we can compute the spinor norm of $\theta_{E,\ell}(\Frob_q)$ by using the special value of $L(T,E)$ arising in the \emph{Birch and Swinnerton-Dyer conjecture}.   We will discuss the Birch and Swinnerton-Dyer conjecture in \S\ref{SS:BSD}.  We shall use known cases to give an simple description of the square class $L(1/q,E) \cdot (\QQ^\times)^2$ when $L(1/q,E)$ is non-zero.

In \S\ref{S:twists}, we follow Hall and consider families of quadratic twists.  We shall constructing a representation that encodes the representations $\theta_{E,\ell}$ as $E$ varies in our family.  In \S\ref{SS:big monodromy Hall}, we state an explicit version of a \emph{big monodromy} result of Hall.

In \S\ref{S:criterion}, we state a criterion to ensure that the representation of \S\ref{S:twists} will produce a group $\Omega(V)$ as the Galois group of a regular extension of $\QQ(t)$.   For any given example, the conditions are straightforward to verify using some basic algebra and Tate's algorithm.

In sections \ref{SS:numerics odd}, \ref{SS:numerics even 1} and \ref{SS:numerics even 2}, we give many examples and use our criterion to prove Theorem~\ref{T:Omega total} for all $n>5$.   We shall not explain how the equations in these sections were found;  they were discovered through many numerical experiments (though the paper \cite{MR1129371} served as a useful starting point since it gives many elliptic surfaces with only four singular fibers).

Finally, in \S\ref{S:N=5} we complete the proof of Theorem~\ref{T:Omega total} for $n=5$.   The big monodromy criterion of Hall does not apply for our example, so we need to prove it directly.\\

\subsection*{Notation}
Throughout the paper, we will feely use Tate's algorithm, see \cite{MR1312368}*{IV~\S9} or \cite{MR0393039}.  All fundamental groups, cohomology and sheaves in the paper are with respect to the \'etale topology.  We will often suppress base points for our fundamental groups, so many groups and representations will only be determined up to conjugacy.  We will indicate base change of schemes by subscripts, for example given a scheme $X$ over $\QQ$, we denote by $X_{\Qbar}$ the corresponding scheme base changed to $\Qbar$.  

\subsection*{Acknowledgements.} 
Numerical experiments and computations were done using \texttt{Magma} \cite{Magma}.  Thanks to Brian Conrad for comments and corrections.

\section{\texorpdfstring{$L$}{L}-functions of elliptic curves over global function fields}  \label{S:elliptic background}
 
In this section, we give some background on the arithmetic of elliptic curves defined over global function fields.

Fix a finite field $k$ whose cardinality $q$ is relatively prime to $2$ and $3$.  Let $C$ be a smooth, proper and geometrically irreducible curve of genus $g$ over $k$ and denote its function field by $K$.   Let $|C|$ be the set of closed points of $C$.  For each $x\in |C|$, let $\FF_x$ be the residue field at $x$ and let $\deg x$ be the degree of $\FF_x$ over $k$.  Each closed point $x\in |C|$ gives a discrete valuation $v_x\colon K \twoheadrightarrow  \ZZ\cup\{\infty\}$ and we denote by $K_x$ the corresponding local field.

Fix an elliptic curve $E$ defined over $K$ whose $j$-invariant $j_E$ is non-constant (i.e., $j_E\in K-k$).   Let $\pi\colon\calE\to C$  be the N\'eron model of $E/K$, cf.~\cite{MR1045822}*{\S1.4}.  Let $U\subseteq C$ be the dense open subset complementary to the finite number of points of bad reduction for $E$.  By abuse of notation, we let $E \to U$ be the (relative) elliptic curve $\pi^{-1}(U)\xrightarrow{\pi} U$; the fiber over the generic point of $U$ is $E/K$.

\subsection{Kodaira symbols}

For each closed point $x$ of $C$,  we can assign a \defi{Kodaira symbol} to the elliptic curve $E$ after base extending by the local field $K_x$; it can be quickly computed using Tate's algorithm, cf.~\cite{MR1312368}*{IV~\S9} or \cite{MR0393039}.     The possible Kodaira symbols are the following: $\I_n$ ($n\geq 0$), $\I_n^*$ ($n\geq 0$), $\II$, $\III$, $\IV$, $\IV^*$, $\III^*$, $\II^*$.     

Let $\Kod(E)$ be the multiset consisting of the Kodaira symbols of $E$ at the points $x\in C$ for which $E$ has bad reduction; we count a Kodaira symbol at $x$ with multiplicity $\deg x$.   Note that the multiset $\Kod(E)$ does not change if we replace $E$ by its base extension to the function field of $C_{k'}$ for any finite extension $k'/k$.

\subsection{Some invariants} \label{SS:some invariants}

We now define some numerical invariants of the curve $E$.  For each $x \in |C|$, we define integers $f_x(E)$, $e_x(E)$, $\gamma_x(E)$, $\lambda_x(E)$, $r_x(E)$ and $b_x(E)$ as in the following table:
\begin{center}
\begin{tabular}{c||c|c|c|c|c|c|c|c|c|c}
Kodaira symbol at $x$ & $\text{I}_0$ & $\text{I}_0^*$ &$\text{I}_n \,(n\geq 1) $& $\text{II}$ & $\text{III}$ &  $\text{IV}$ &  $\text{I}_n^*\, (n\geq 1)$ &  $\text{IV}^*$ & $\text{III}^*$  & $\text{II}^*$ \\\hline 
$f_x$ & $0$ & $2$ & $1$ & $2$ & $2$ & $2$ & $2$     & $2$ & $2$ & $2$ \\
$e_x$ & $0$ & $6$ &$n$ & $2$ & $3$ & $4$ & $6+n$ & $8$ & $9$  & $10$\\
$\gamma_x$ & $1$ & $1$ & $n/\gcd(2,n)$ & $1$ & $1$ & $3$ & $2/\gcd(2,n)$ & $3$ & $1$  & $1$\\
$\lambda_x$ & $1$ & $1$ & $n$ & $1$ & $1$ & $1$ & $n$ & $1$ & $1$ & $1$\\
$r_x$ & $1$ & $1$ & $1$ & $1$ & $2$ & $3$ & $1$ & $3$ & $2$ & $1$ \\
$b_x$ & $0$ & $0$ & $0$ & $1$ & $1$ & $1$ & $1$ & $1$ & $1$ & $1$ \\
\end{tabular}
\end{center}
Define 
 \[
N_E= -4+4g+\sum_{x\in |C|} f_x(E) \deg(x),\quad \chi_E = \frac{1}{12} \sum_{x\in |C|} e_x(E) \deg x, \quad \text{and}\quad \gamma_E = \prod_{x\in |C|} \gamma_x(E)^{\deg x}.
\]
Define $\calL_E$ to be the product of the primes $\ell \geq 5$ that divide $\lambda_x(E)$ for some $x\in |C|$; it is also the product of primes $\ell\geq 5$ that divide $\max\{1,-v_x(j_E)\}$ for some $x\in |C|$.

\begin{remark} \label{R:Phi dependence}
\begin{romanenum}
\item \label{R:Phi dependence i}
The integers $N_E$, $\chi_E$,  $\gamma_E$ and $\calL_E$ can all be determined directly from the multiset $\Kod(E)$.   
\item
One way to prove that $\chi_E$ is  an integer is to show that it agrees with the Euler characteristic of the sheaf $\OO_X$, where $X\to C$ is a relatively minimal elliptic surface that extends $E\to U$ with $X$ smooth and projective.
\end{romanenum}
\end{remark}
 
Let $\calE_x/\FF_x$ be the fiber of the N\'eron model $\calE\to C$ of $E/K$ at $x$.   We define $c_x(E)$ to be the order of the group $\calE_x(\FF_x)/\calE_x^\circ(\FF_x)$, where $\calE_x^\circ$ is the identity component of the group scheme $\calE_x$.  Define the integer 
 \[
 c_E= \prod_{x\in |C|} c_x(E);
 \]
it is well-defined since $c_x(E)=1$ whenever $E$ has good reduction at $x$.   The integer $c_E$ serves as a ``fudge factor'' in the Birch and Swinnerton-Dyer conjecture, cf.~\S\ref{SS:BSD}, and can be quickly computed using Tate's algorithm.\\

Let $A_E$ be the set of closed points $x$ of $C$ for which $E$ has bad reduction of additive type.  For each $x\in A_E$, let $\chi_x\colon \FF_x^\times \to \{\pm 1\}$ be the non-trivial quadratic character (recall that $q$ is odd).   Let $m_+$ be the number of closed points $x$ of $C$ for which $E$ has split multiplicative reduction.  Define 
\[
\displaystyle\varepsilon_E := (-1)^{m_+} {\prod}_{x\in A_E} \chi_x(-r_x(E)) \in \{\pm 1\};
\]
it is the \defi{root number} of $E$, cf.~Theorem~\ref{T:L-function fundamentals}.

\subsection{\texorpdfstring{$L$}{L}-functions} \label{SS:Lfunction}
Take any closed point $x$ of $C$.  If $E$ has good reduction at $x$, define the polynomial $L_x(T):= 1 - a_x(E) T^{\deg x} + q^{\deg x} T^{2\deg x}$, where $E_x/\FF_x$ is the fiber of $E$ over $x$ and  $a_x(E):=q^{\deg x} +1 -|E_x(\FF_x)|$.    If $E$ has bad reduction at $x$, define $a_x(E)$ to be $1$, $-1$ or $0$ when $E$ has split multiplicative, non-split multiplicative or additive reduction, respectively, at $x$; define the polynomial $L_x(T)=1-a_x(E) T^{\deg x}$.    

The \defi{$L$-function} of $E$ is the formal power series
\[
L(T,E) := {\prod}_{x\in |C|} L_x(T)^{-1} \in \ZZ[\![T]\!].
\]
The following gives some fundamental properties of $L(T,E)$; we will give a sketch in \S\ref{SS:proof theta mod ell}.
 
\begin{theorem} \label{T:L-function fundamentals}
The $L$-function $L(T,E)$ is a polynomial of degree $N_E$ with integer coefficients and satisfies the functional equation
\[
T^{N_E} L(T^{-1}/q, E) = \varepsilon_E \cdot L(T/q,E).
\]
\end{theorem}

If $q$ is a power of $2$ or $3$, then Theorem~\ref{T:L-function fundamentals} will hold except the numerical recipes for $N_E$ and $\varepsilon_E$ need to be refined.  One can also consider the case where $E$ has constant $j$-invariant; $L(T,E)$ is still a rational function but is no longer a polynomial.

Another important property of $L(T,E)$, which we will not require, is that all its complex roots have absolute value $q^{-1}$.  This will follow from the cohomological interpretation and the work of Deligne.

\subsection{The Birch and Swinnerton-Dyer conjecture}   \label{SS:BSD}

The \defi{Mordell-Weil theorem} for $E$ says that the abelian group $E(K)$ is finitely generated.   It is straightforward to compute the torsion subgroup of $E(K)$ but computing its rank is more difficult.    Before stating the conjecture, we mention several invariants of $E$:
\begin{itemize}
 \item 
 Let $\Sha_E$ be the Tate-Shafarevich group of $E/K$.  
 \item
 Let $\mathcal{D}_E$ be the minimal discriminant of $E/K$; we may view it as a divisor of $C$.  Using Tate's algorithm (and our ongoing assumption that $q$ is not a power of $2$ or $3$), we find that the degree of $\mathcal{D}_E$ is $12\chi_E$.
 \item
 The integer $c_E$ from \S\ref{SS:some invariants}.
\item
Let $E(K)_{\tors}$ be the torsion subgroup of $E(K)$. 
\item
The \defi{regulator} of $E$ is the real number $R_E:= \det( \ang{P_i}{P_j} )$, where $\ang{\,}{\,}\colon E(K)\times E(K) \to \RR$ is the canonical height pairing and $P_1,\ldots, P_r \in E(K)$ are points that give a basis for the free abelian group $E(K)/E(K)_{\tors}$.  
\end{itemize}
 
The following is a conjectural relation between the rank of the Mordell-Weil group $E(K)$ and its $L$-function.
 
\begin{conjecture}[Birch and Swinnerton-Dyer] \label{C:BSD}
Let $r$ be the rank of $E(K)$.
\begin{alphenum}
\item \label{C:BSD a}
The rank $r$ agrees with multiplicity of $1/q$ as a root of $L(T,E)$.   
\item \label{C:BSD b}
For some prime $\ell$, the $\ell$-primary component of $\Sha_E$ is finite.
\item \label{C:BSD c}
The group $\Sha_E$ is finite and 
\[
L(q^{-s}, E) \sim \frac{|\Sha_E|\, R_E\, c_E }{|E(K)_{\tors}|^2\cdot q^{g-1+\chi_E}}  \cdot (s-1)^r
\]
as $s\to 1$.
\end{alphenum}
\end{conjecture}

A nice exposition of the conjecture, with the explicit constant $\alpha$, is given by Gross in \cite{MR2882691}; he also gives similar details for the more familiar number field analogue.  Note that the regulator in \cite{MR2882691} is equal to $R_E/|E(K)_{\tors}|^2$.

Conjecture~\ref{C:BSD}(\ref{C:BSD c}) clearly implies the other two parts; they are in fact equivalent.
 
\begin{theorem}[Artin-Tate, Milne] \label{T:Artin-Tate-Milne}
Statements (\ref{C:BSD a}), (\ref{C:BSD b}) and (\ref{C:BSD c}) of Conjecture~\ref{C:BSD} are equivalent.
\end{theorem}
\begin{proof}
This follows from Theorem~8.1 of \cite{MR0414558}; it builds on the work of Artin and Tate presented in Tate's 1966 Bourbaki seminar \cite{Tate1964-1966}.   It should be noted that the $L$-functions in \cite{Tate1964-1966} do not include the now familiar factors at the bad places; those were later worked out by Tate in \cite{MR0393039}*{\S5}.   
\end{proof}

The following gives an a priori inequality between the two quantities in Conjecture~\ref{C:BSD}(\ref{C:BSD a}); it follows from the injectivity of the homomorphism $h$ of Theorem~5.2 of \cite{Tate1964-1966}.   

\begin{proposition}\label{P:rank inequality}
The rank of $E(K)$ is always less than or equal to the multiplicity of $1/q$ as a root of $L(T,E)$.
\end{proposition}

In this paper, we will use only the following special consequence of the above results.

\begin{corollary} \label{C:BSD main}
Suppose that $L(1/q,E) \neq 0$.  Then 
\[
L(1/q,E)\in q^{g-1+\chi_E}  c_E \cdot (\QQ^\times)^2.
\]
\end{corollary}
\begin{proof}
Let $r$ be the rank of $E(K)$.   Proposition~\ref{P:rank inequality} and our assumption $L(1/q,E)\neq 0$ implies that $r\leq 0$.  Therefore, $r=0$, i.e., $E(K)$ is  finite.   Since Conjecture~\ref{C:BSD}(\ref{C:BSD a}) holds, Theorem~\ref{T:Artin-Tate-Milne} implies that the group $\Sha_E$ is finite and that
\[
L(1/q,E)= \frac{|\Sha_E|\, \cdot c_E }{ |E(K)_{\tors}|^2\cdot q^{g-1+\chi_E}}
\]
(we have $R_E=1$ since $r=0$).  Therefore,
\[
L(1/q,E) \in |\Sha_E| \cdot q^{g-1+\chi_E} \cdot c_E \cdot  (\QQ^\times)^2.
\]
Cassels constructed an alternating and non-degenerate pairing $\Sha_E\times \Sha_E \to \QQ/\ZZ$, cf.~\cite{MR2261462}*{Theorem~6.13}, from which one can deduce that $\Sha_E$ has square cardinality (we are using of course that $\Sha_E$ is finite in our case).   The result is then immediate.
\end{proof}

\subsection{\texorpdfstring{$L$}{L}-functions modulo \texorpdfstring{$\ell$}{l}}  \label{SS:Lfunctions mod ell}

Take any prime $\ell \nmid 6q\calL_E$.    Let $E[\ell]$ be the $\ell$-torsion subscheme of $E$; it is a sheaf of $\FF_\ell$-modules on $U$ that is free of rank $2$.  The lisse sheaf $E[\ell]$ corresponds to a representation
\[
\rho_{E,\ell} \colon \pi_1(U,\bbar\eta) \to \Aut_{\FF_\ell}(E[\ell]_{\bbar\eta}) \cong \GL_2(\FF_\ell),
\]
where $\bbar\eta$ is a geometric generic point of $U$.   The Weil pairing $E[\ell] \times E[\ell] \to \FF_\ell(1)$ is non-degenerate and alternating, so $\rho_{E,\ell}(\pi_1(U_{\bbar{k}})) \subseteq \SL_2(\FF_\ell)$.   The following ``big monodromy'' result will be proved in \S\ref{SS:proof elliptic big monodromy}.

\begin{proposition}  \label{P:elliptic big monodromy}
We have $\rho_{E,\ell}(\pi_1(U_{\bbar{k}})) = \SL_2(\FF_\ell)$.
\end{proposition}

Pushing forward, we obtain a sheaf $j_*(E[\ell])$ of $\FF_\ell$-modules on $C$, where $j\colon U \to C$ is the inclusion morphism.   Define the $\FF_\ell$-vector space 
\[
V_{E,\ell} := H^1(C_{\bbar{k}}, j_*(E[\ell])).
\]
The Weil pairing $E[\ell] \times E[\ell] \to \FF_\ell(1)$ is non-degenerate and alternating, and  gives rise to an isomorphism $E[\ell]^\vee(1) \cong E[\ell]$ of sheaves.  Using this isomorphism and Poincar\'e duality (for example, as in \cite{MR559531}*{V Proposition~2.2(b)}), we obtain a non-degenerate and symmetric pairing
\[
\ang{\,}{\,}\colon V_{E,\ell} \times V_{E,\ell} \to H^2(C_{\bbar{k}}, \FF_\ell(1)) \cong \FF_\ell.
\]
Therefore, $V_{E,\ell}$ with the pairing $\ang{\,}{\,}$ is an orthogonal space over $\FF_\ell$.

There is a natural action of $\Gal(\bbar{k}/k)$ on the vector space $V_{E,\ell}$.   The Galois action on $V_{E,\ell}$ respects the pairing and hence gives rise to a representation
\[
\theta_{E,\ell}\colon \Gal(\bbar{k}/k) \to \Or(V_{E,\ell}).
\]
Let $\Frob_q \in\Gal(\bbar{k}/k)$ be the \emph{geometric} Frobenius (i.e., the inverse of $x\mapsto x^q$).   The following says that we can recover $L(T,E)$ modulo $\ell$ from the characteristic polynomial of $\theta_{E,\ell}(\Frob_q)$.

\begin{proposition} \label{P:main theta mod ell}
The vector space $V_{E,\ell}$ has dimension $N_E$ over $\FF_\ell$ and 
\[
\det(I - \theta_{E,\ell}(\Frob_q)T) \equiv L(T/q,E) \pmod{\ell}.
\]
\end{proposition}

In many cases, the following proposition gives a way to compute the determinant and the spinor norm of $\theta_{E,\ell}(\Frob_q)$ in terms of some of the invariants of $E$.   Our assumption $\ell \nmid 6q \calL_E$ ensures that $\ell$ does not divide the integer $2^{N_E} q^{g-1+\chi_E} c_E \gamma_E$ (if $c_x(E)$ is divisible by a prime $p\geq 5$ for some $x\in|C|$, then Tate's algorithm shows that $E$ must have Kodaira symbol $\I_n$ at $x$ where $n=c_x(E)$, and hence $p$ divides $\calL_E$).

\begin{proposition} \label{P:theta comp}  
\begin{romanenum}
\item \label{P:theta comp i}
We have $\det(\theta_{E,\ell}(\Frob_q)) = (-1)^{N_E}\cdot \varepsilon_E$.
\item \label{P:theta comp ii}
If $\det(I-\theta_{E,\ell}(\Frob_q))\neq 0$, then $\spin(-\theta_{E,\ell}(\Frob_q)) = 2^{N_E} q^{g-1+\chi_E} \,c_E \cdot (\FF_\ell^\times)^2$.
\item \label{P:theta comp iii}
If $\det(I+\theta_{E,\ell}(\Frob_q))\neq 0$, then $\spin(\theta_{E,\ell}(\Frob_q)) = 2^{N_E} q^{g-1+\chi_E} c_E \cdot\gamma_E \cdot (\FF_\ell^\times)^2$.
\item \label{P:theta comp iv}
If $\det(I\pm \theta_{E,\ell}(\Frob_q))\neq 0$, then $\disc(V_{E,\ell}) = \gamma_E \cdot (\FF_\ell^\times)^2$.
\end{romanenum}
\end{proposition}

The proofs of Propositions~\ref{P:main theta mod ell} and \ref{P:theta comp} will be given in \S\ref{SS:proof theta mod ell} and \S\ref{SS:proof of theta comp}, respectively.  

\begin{remark} \label{R:alternate Vs}
Let $\calE[\ell]$ be the $\ell$-torsion subscheme of the N\'eron model $\calE\to C$; it is a sheaf of $\FF_\ell$-modules on $C$.   For any non-empty open subvariety $U'$ of $C$, one can show that $\calE[\ell]$ is canonically isomorphic to $j'_* j'^*(\calE[\ell])$, where $j'\colon U'\hookrightarrow C$ is the inclusion morphism.    In particular with $U'=U$, we find that $V_{E,\ell}=H^1(C_{\bbar{k}}, \calE[\ell])$.    If $U' \subseteq U$, then we have $V_{E,\ell} = H^1(C_{\bbar{k}}, j'_*(E[\ell]|_{U'}))$.
\end{remark}

\begin{remark} 
Let $X\to C$ be a relatively minimal morphism extending $E\to U$, where $X$ is a smooth and projective surface over $k$.   One can give a filtration of $H^2(X,\FF_\ell(1))$ as an $\FF_\ell[\Gal(\bbar{k}/k)]$-module such that one of the quotients is $V_{E,\ell}$ (and the cup pairing on $H^2$ induces our pairing on $V_{E,\ell}$).   Similar remarks hold for the more general constructions of \S\ref{S:twists}.
\end{remark}

\subsection{Proof of Proposition~\ref{P:elliptic big monodromy}} \label{SS:proof elliptic big monodromy}

Take any proper subgroup $H$ of $\SL_2(\ZZ/\ell\ZZ)$.  For a fixed algebraically closed field $F$ whose characteristic is not $\ell$, let $X(\ell)$ be the modular curve over $F$ parametrizing elliptic curves with level $\ell$-structure; it is smooth and projective.   There is a natural action of $\SL_2(\ZZ/\ell\ZZ)$ on $X(\ell)$.  Define the curve $X_H=X(\ell)/H$ and let $\pi_H\colon X_H\to X(\ell)/\SL_2(\ZZ/\ell\ZZ)\cong\PP^1_F$ be the morphism down to the $j$-line.   Let $m$ be the least common multiple of the order of the poles of $\pi_H$. 

We claim that $m=\ell$.   There is a model of the modular curve $X_H$ over $\Spec \ZZ[1/\ell]$ such that the the divisor consisting of cusps is \'etale over $\Spec \ZZ[1/\ell]$, for background see \S3 in part IV of \cite{MR0337993}.   The integer $m$ is thus independent of $F$, so we may take $F=\CC$.   Let $\Gamma$ be the congruence subgroup consisting of matrices $A\in \SL_2(\ZZ/\ell\ZZ)$ such that $A$ modulo $\ell$ belongs to $H$.   The map $\pi_H\colon X_H(\CC)\to \PP^1(\CC)$ of compact Riemann surfaces comes from compactifying the natural quotient map $\mathfrak{h}/\Gamma \to \mathfrak{h}/\SL_2(\ZZ)$, where $\SL_2(\ZZ)$ acts on the upper-half plane $\mathfrak{h}$ via linear fractional transformations.   Therefore, $m$ is equal to the least common multiple of the width of the cusps of $\Gamma$.   Since $\Gamma$ has level $\ell$, we have $m=\ell$  (in \cite{MR0167533}, the quantity $m$ is called the ``general level'' of $\Gamma$ and it is shown to agree with the usual level).

We now focus on the case $F=\bbar{k}$.  Suppose that $H=\rho_{E,\ell}(\pi_1(U_{\bbar{k}}))$ is a proper subgroup of $\SL_2(\ZZ/\ell\ZZ)$.    Let $J\colon U_{\bbar{k}} \to \Aff^1_{\bbar{k}}$ be the morphism given by the $j$-invariant of $E$; it is dominant since $E$ is non-isotrivial.     Since $\rho_{E,\ell}(\pi_1(U_{\bbar{k}}))\subseteq H$, the morphism $J$ factors as 
\[
U_{\bbar{k}} \to X_H \xrightarrow{\pi_H} \Aff^1_{\bbar{k}}.
\]     
Let $m'$ be the least common multiple of the order of the poles of the morphism $C_{\bbar{k}}\to \PP^1_{\bbar{k}}$ extending $J$.  The integer $m'$ is divisible by $m$; the least common multiple of the order of the poles of $\pi_H$.   By our claim, $m'$ is divisible by $\ell$.  However, $\ell$ dividing $m'$ implies that there is a closed point $x$ of $C$ such that $v_x(j_E)$ is negative and divisible by our prime $\ell\geq 5$; this in turn implies that $\ell$ divides $\calL_E$.   This contradicts our ongoing assumption that $\ell\nmid \calL_E$, so $H=\SL_2(\ZZ/\ell\ZZ)$ as desired.

\subsection{Proof of Proposition~\ref{P:main theta mod ell} and Theorem~\ref{T:L-function fundamentals}}  \label{SS:proof theta mod ell}

We first recall a cohomological description of $L(T,E)$.    For each integer $n\geq 1$, let $E[\ell^n]$ be the $\ell^n$-torsion subscheme of $E$; it is a lisse sheaf of $\ZZ/\ell^n\ZZ$-modules on $U$ that is free of rank $2$.    The sheaves $\{E[\ell^n]\}_{n\geq 1}$ with the multiplication by $\ell$ morphism $E[\ell^{n+1}]\to E[\ell^n]$ form a lisse sheaf of $\ZZ_\ell$-modules on $U$ which we denote by $T_\ell(E)$.    Define the $\QQ_\ell$-sheaf $\calF:=j_*(T_\ell(E))\otimes_{\ZZ_\ell} \QQ_\ell$, where $j\colon U\hookrightarrow C$ is the inclusion morphism, and let $\calF^\vee$ be its dual.  

Take any closed point $x$ of $C$ and let $\bbar{x}$ be a geometric point of $C$ mapping to $x$ arising from a choice of algebraic closure $\FFbar_x$ of $\FF_x$.   The geometric Frobenius $\Frob_x \in\Gal(\FFbar_x/\FF_x)$ acts on the fibers $\calF_{\bbar{x}}$ and $\calF^\vee_{\bbar{x}}$.    One can show that
\[
\det(I - \Frob_x T^{\deg x} \, |\, \calF^\vee_{\bbar{x}})=L_x(T).
\]   
The Weil pairings give an isomorphism $\calF^\vee\cong \calF(-1)$, and hence $\det(I - \Frob_x T^{\deg x} \, |\, \calF_{\bbar{x}}) = L_x(T/q)$.  Therefore, 
\[
L(T/q,E) = \prod_{x\in |C|} \det(I - \Frob_x T^{\deg x} \, |\, \calF_{\bbar{x}})^{-1}.
\]  
By the Grothendieck-Lefschetz trace formula, we have
\[
L(T/q,E)= {\prod}_{i} \det\big(I - \Frob_q T \,|\, H^i(C_{\bbar{k}}, \calF) \big)^{(-1)^{i+1}}.
\]
Lemma~\ref{L:Hi vanishing}(\ref{L:Hi vanishing ii}) below then shows that $L(T/q,E)$ is equal to the polynomial 
\[
\det(I-\Frob_q T \,| \, M \otimes_{\ZZ_\ell} \QQ_\ell),   
\]
where $M$ is the $\ZZ_\ell$-module $H^1(C_{\bbar{k}},j_*(T_\ell(E)))$.  That the polynomial $L(T,E)$ has integer coefficients is clear from its power series definition.

\begin{lemma} \label{L:Hi vanishing}
Take any integer $i\neq 1$.
\begin{romanenum} 
\item \label{L:Hi vanishing i}
We have $H^i(C_{\bbar{k}},j_*(E[\ell^n]))=0$ for all $n\geq 1$.   
\item \label{L:Hi vanishing ii}
We have $H^i(C_{\bbar{k}},j_*(T_\ell(E)))=0$ and $H^i(C_{\bbar{k}},\calF)=0$.
\end{romanenum}
\end{lemma}
\begin{proof}
The lisse sheaf $E[\ell^n]$ corresponds to a representation
\[
\rho_{E,\ell^n} \colon \pi_1(U,\bbar\eta) \to \Aut_{\ZZ/\ell^n\ZZ}(E[\ell^n]_{\bbar\eta}) \cong \GL_2(\ZZ/\ell^n\ZZ),
\]
where $\bbar\eta$ is a geometric generic point of $U$.   The Weil pairing on $E[\ell^n]$ is non-degenerate and alternating, so $H:=\rho_{E,\ell^n}(\pi_1(U_{\bbar{k}})) \subseteq \SL_2(\ZZ/\ell^n\ZZ)$.    Proposition~\ref{P:elliptic big monodromy} implies that the image of $H$ modulo $\ell$ is $\SL_2(\ZZ/\ell\ZZ)$.  We thus have $H=\SL_2(\ZZ/\ell^n\ZZ)$ since $\SL_2(\ZZ/\ell^n\ZZ)$ has no proper subgroups whose image modulo $\ell$ is $\SL_2(\ZZ/\ell\ZZ)$, cf.~Lemma~2 on page IV-23 of \cite{MR0263823}.  

We now prove (\ref{L:Hi vanishing i}).  Since $C$ has dimension $1$, we need only consider $i\in \{0,2\}$.  The Weil pairing on $E[\ell^n]$ gives rise to an isomorphism $E[\ell^n]^\vee(1)\cong E[\ell^n]$ of sheaves on $U$.  Using this isomorphism and Poincar\'e duality (for example, as in \cite{MR559531}*{V Proposition~2.2(b)}), we obtain a non-degenerate pairing $H^0(C_{\bbar{k}}, j_*(E[\ell^n])) \times H^2(C_{\bbar{k}}, j_*(E[\ell^n]))) \to \ZZ/\ell^n\ZZ$.  So we may assume that $i=0$.   We have
\[
H^0(C_{\bbar{k}}, j_*(E[\ell^n])) = H^0(U_{\FFbar_q}, E[\ell^n]) = (E[\ell^n]_{\bbar{\eta}})^{\pi_1(U_{\bbar{k}},\bbar\eta)}=0,
\]
where the last equality uses that $\rho_{E,\ell^n}(\pi_1(U_{\bbar{k}})) =\SL_2(\ZZ/\ell^n\ZZ)$.
\end{proof}	
  
We have a short exact sequence $0 \to T_\ell(E) \xrightarrow{\times \ell} T_\ell(E) \to E[\ell] \to 0$ of sheaves on $U$.  Pushing forward, we have a short exact sequence 
\[
0 \to j_*(T_\ell(E)) \xrightarrow{\times \ell} j_*(T_\ell(E)) \to j_*(E[\ell]) \to 0
\]
of sheaves on $C$ which gives an exact sequence
\begin{equation} \label{E:M seq}
0 = H^0(C_{\bbar{k}}, j_*(E[\ell])) \to M \xrightarrow{\times \ell} M \to V_{E,\ell} \to H^2(C_{\bbar{k}}, j_*(T_\ell(E)))=0,
\end{equation}
where we have used Lemma~\ref{L:Hi vanishing} for the $H^0$ and $H^2$ terms.  From (\ref{E:M seq}), the finitely generated $\ZZ_\ell$-module $M$ has trivial $\ell$-torsion and is thus a free $\ZZ_\ell$-module of finite rank.   From (\ref{E:M seq}), we have an isomorphism of $M/\ell M$ and $V_{E,\ell}$ that respects the action of $\Frob_q$.    Therefore, $L(T/q,E)= \det(I - \Frob_q T | M\otimes_{\ZZ_\ell} \QQ_\ell)$ is congruent modulo $\ell$ to $\det(I-\Frob_q T | V_{E,\ell})$.  \\

To complete the proof of Proposition~\ref{P:main theta mod ell}, it remains to show that $V_{E,\ell}$ has dimension $N_E$ over $\FF_\ell$.  As a consequence, we will deduce that $L(T,E)$ has degree $N_E$.  Define $\chi_\ell = \sum_{i} (-1)^i \dim_{\FF_\ell} H^i(C_{\bbar{k}},j_*(E[\ell]))$.    By \cite{MR559531}*{V Theorem~2.12}, we have
\[
\chi_\ell = (2-2g) \dim_{\FF_\ell} j_*(E[\ell])_{\bbar\eta} - {\sum}_x \mathfrak{f}_x = 4-4g  - {\sum}_x \mathfrak{f}_x,
\]
where the sums are over the closed points of $C_{\bbar{k}}$ and $\mathfrak{f}_x$ is the (exponent of the) conductor of the sheaf $j_*(E[\ell])$ at $x$.   Since the sheaf $j_*(E[\ell])$ is tamely ramified ($q$ is not a power of $2$ or $3$), we have $\mathfrak{f}_x = \dim_{\FF_\ell} j_*(E[\ell])_{\bbar\eta}  - \dim_{\FF_\ell} j_*(E[\ell])_{\bbar{x}} = 2 - \dim_{\FF} j_*(E[\ell])_{\bbar{x}}$.  In particular, $\mathfrak{f}_x$ is $0$, $1$, or $2$ if $E$ has good, multiplicative or additive reduction, respectively, at $x$.   The sum of the $\mathfrak{f}_x$ over the closed points $x$ of $C_{\bbar{k}}$ is equal to $\sum_{x\in |C|} f_x(E) \deg x$.   Therefore, $-\chi_\ell$ equals $N_E$.    Using Lemma~\ref{L:Hi vanishing}(\ref{L:Hi vanishing i}), we deduce that $V_{E,\ell}=H^1(C_{\bbar{k}},j_*(E[\ell]))$ has dimension $-\chi_\ell=N_E$ over $\FF_\ell$.\\

Let us prove the functional equation for $L(T,E)$ using what we have already proved.  Take any prime $\ell\nmid 6q\calL_E$.   We have shown that $L(T/q,E)\equiv \det(I-AT) \pmod{\ell}$ for some $A\in \Or(V_{E,\ell})$.   We have $T^{N_E} \det(I-AT^{-1}) = \pm \det(I-AT)$ for every $A\in \Or(V_{E,\ell})$, so $T^{N_E} L(T^{-1}/q,E) \equiv \pm L(T,E) \pmod{\ell}$.   Since this holds for all but finitely many primes $\ell$, we must have $T^{N_E} L(T^{-1}/q,E) = \varepsilon L(T/q,E)$ for a unique $\varepsilon\in \{\pm 1\}$.  

We can express $\varepsilon$ as a product of local root numbers $\varepsilon_x(E)$ over the closed points $x$ of $C$; note that $\varepsilon_x(E)=1$ if $E$ has good reduction at $x$.     Fix a closed point $x$ of $C$ for which $E$ has bad reduction and let $\kappa$ be the Kodaira symbol of $E$ at $x$.  If $\kappa$ is not of the form $\I_n$ or $\I_n^*$ with $n>0$, then $\varepsilon_x(E)=\chi_x(-r_x(E))$ by Theorem~3.1 of \cite{MR2183392}; this uses the ongoing assumption that $\gcd(q,6)=1$ and also that the $e$ of loc.~cit.~is $12/\gcd(e_x(E),12)$.   If $\kappa$ is of the form $\I_n^*$ for some $n\geq 0$, then $\varepsilon_x(E)=\chi_x(-r_x(E))$ by Theorem~3.1(2) of \cite{MR2183392}.  If $\kappa$ is of the form $\I_n$ for some $n>0$, then $\varepsilon_x(E)$ is $-1$ or $1$ when $E$ has split or non-split multiplicative reduction, respectively, at $x$; cf.~Theorem~3.1(2) and Lemma~2.2 of \cite{MR2183392}.     This shows that $\varepsilon$ agrees with our value $\varepsilon_E$.  This completes the proof of Theorem~\ref{T:L-function fundamentals}.

\subsection{Proof of Proposition~\ref{P:theta comp}} \label{SS:proof of theta comp}
To compute spinor norms, we will use the following result of Zassenhaus.

\begin{lemma} \label{L:Zassenhaus}
Let $V$ be an orthogonal space of dimension $N$ defined over a finite field $\FF$ of odd characteristic.   If $B\in \Or(V)$ satisfies $\det(I+B)\neq 0$, then $\spin(B) = 2^N \det(I+B) \cdot (\FF^\times)^2$.
\end{lemma}
\begin{proof}
This is a special case of Zassenhaus' formula for the spinor norm in \S2 of \cite{MR0148760}; see Theorem~C.5.7 of \cite{Conrad-Zass} for a modern proof.
\end{proof}

Set $A:=\theta_{E,\ell}(\Frob_q)$.    By  Proposition~\ref{P:main theta mod ell}, the vector space $V_{E,\ell}$ has dimension $N_E$ and we have
\begin{equation} \label{E:g and L}
 \det(I-AT)  \equiv L(T/q,E) \pmod{\ell}.
\end{equation}
Since $A$ belongs to $\Or(V_\ell)$, we find that  $T^{N_E} \det(I-AT^{-1}) = \det(-A) \det(I-AT)$.  By (\ref{E:g and L}) and the functional equation in Theorem~\ref{T:L-function fundamentals}, we also have $T^{N_E} \det(I-AT^{-1}) = \varepsilon_E \det(I-AT)$.  Comparing these two equations, we deduce that $\det(-A)=(-1)^{N_E} \det(A)$ agrees with $\varepsilon_E$.   This proves part (\ref{P:theta comp i}).

Now suppose that $\det(I-A)\neq 0$.   Since $\det(I+(-A))\neq 0$, Lemma~\ref{L:Zassenhaus} and (\ref{E:g and L}) give us
\[
\spin(-A) =2^{N_E} \det(I-A) \cdot(\FF_\ell^\times)^2 = 2^{N_E} L(1/q,E) \cdot(\FF_\ell^\times)^2.
\]
Therefore, $\spin(-A) =  2^{N_E} q^{g-1+\chi_E} c_E \cdot (\FF_\ell^\times)^2$ by Corollary~\ref{C:BSD main}  (as noted in \S\ref{SS:Lfunctions mod ell}, $\ell \nmid q^{g-1+\chi_E} c_E$).   This proves (\ref{P:theta comp ii}).  Before proving parts (\ref{P:theta comp iii}) and (\ref{P:theta comp iv}), we need the following lemma.

\begin{lemma} \label{L:technical gamma}
Let $E'/K$ be the quadratic twist of $E/K$ by a non-square $\beta$ in $k^\times$.   Take any closed point $x$ of $C$.
\begin{romanenum}
\item \label{L:technical gamma a}
The curves $E'$ and $E$ have the same Kodaira symbol at $x$.
\item \label{L:technical gamma b}
We have $a_x(E')=(-1)^{\deg x} a_x(E)$.
\item \label{L:technical gamma c}
The integer $c_x(E)c_x(E') \gamma_x(E)^{\deg x}$ is a square.
\end{romanenum}
\end{lemma}
\begin{proof}
First suppose that $\deg x$ is even.    Since $\beta$ is a square in $\FF_x$, we find that $E$ and $E'$ are isomorphic over $K_x$.  All the parts of the lemma are now immediate.

Now suppose that $\deg x$ is odd.   Let $\OO_x$ be the valuation ring of $K_x$ and let $\pi$ be a uniformizer. Tate's algorithm, as presented in \cite{MR1312368}*{IV~\S9} or \cite{MR0393039}, starts with a Weierstrass equation
\begin{equation} \label{E:technical gamma 1}
\yy^2+a_1 \xx \yy +a_3 y = \xx^3 + a_2 \xx^2 +a_4 \xx +a_6
\end{equation}
for $E$ over the local field $K_x$ with $a_i\in \OO_x$.    The algorithm then changes coordinates several times which imposes various conditions on the powers of $\pi$ dividing the coefficients $a_i$; these conditions in loc.~cit. are boxed (similar remarks hold for Subprocedure~7).   By completing the square in (\ref{E:technical gamma 1}), we find that this elliptic curve is isomorphic to the one defined by the equation
\begin{equation} \label{E:technical gamma 2}
\yy^2 = \xx^3 + a_2' \xx^2 +a_4' \xx +a_6',
\end{equation}
with $a_2' = a_2+a_1^2/4$, $a_4'=a_4+a_1 a_3/2$ and $a_6'=a_6+a_3^2/4$.   Using that $q$ is odd, we find that $a_2'$, $a_4'$ and $a_6'$ belong to $\OO_x$ and that the conditions in Tate's algorithm are preserved.   So, we may thus always assume in any application of Tate's algorithm that we have a Weierstrass equation of the form (\ref{E:technical gamma 2}).

If $E$ over $K_x$ is given by (\ref{E:technical gamma 2}), then $E'$ over $K_x$ has a Weierstrass equation $\yy^2 = \xx^3 + a_2' \beta \xx^2 +a_4' \beta^2\xx +\beta^3 a_6'$.   Applying Tate's algorithm, it is now easy to see that $E$ and $E'$ have the same Kodaira symbol and to determine the possibilities for $c_x(E)c_x(E')$.   Let $\kappa$ be the Kodaira symbol of $E$ and $E'$ at $x$.  If $\kappa \in \{\I_0,\II,\II^*\}$, then $c_x(E)c_x(E')=1$.  If $\kappa=\I_n$ with $n>0$, then $c_x(E)c_x(E') = \gcd(2n) \cdot n$ (precisely one of curves $E$ and $E'$ has split reduction at $x$; this uses that $\beta$ is a non-square in $\FF_x$ since $\deg x$ is odd).  If $\kappa \in \{\III,\III^*\}$, then $c_x(E)c_x(E')=2^2$.   If $\kappa \in \{\IV,\IV^*\}$, then $c_x(E)c_x(E')=1\cdot 3=3$.   If $\kappa=\I_n^*$ with $n$ odd, then $c_x(E)c_x(E')=2\cdot 4=8$.   If $\kappa=\I_n^*$ with $n>0$ even, then $c_x(E)c_x(E') \in \{2^2,4^2\}$.  Finally, if $\kappa=\I_0^*$, then $c_x(E)c_x(E') \in \{1^2,2^2,4^2\}$.    In all these cases, we find that integer $c_x(E)c_x(E') \gamma_x(E)$ is a square.   Since $\deg x$ is odd, we conclude that $c_x(E)c_x(E') \gamma_x(E)^{\deg x}=c_x(E)c_x(E') \gamma_x(E)\cdot (\gamma_x(E)^{(\deg x-1)/2})^2$ is a square.

It remains to verify that $a_x(E')=-a_x(E)$.  This is immediate if $E$, and hence $E'$, has additive reduction at $x$ since $a_x(E')=a_x(E)=0$.  If $E$, and hence $E'$, has multiplicative reduction at $x$, then one has split reduction and the other non-split reduction (since $\beta$ is not a square in $\FF_x$), so $a_x(E')=-a_x(E)$.    

Finally suppose that $E$, and hence $E'$, has good reduction at $x$.  Fix a Weierstrass model $\yy^2=f(\defi{x})$ for $E_x/\FF_x$ with a cubic $f\in \FF_x[\defi{x}]$; the equation $\beta \yy^2=f(\defi{x})$ is a model of $E'_x/\FF_x$.   Take any $a\in \FF_x$.  If $f(a)$ is a non-zero square in $\FF_x$, then there are two point in $E_x(\FF_x)$ with $\defi{x}$-coordinate $a$.  If $f(a)$ is a non-square in $\FF_x$, then there are two point on $E'_x(\FF_x)$ with $\defi{x}$-coordinate $a$.  If $f(a)=0$, then $E_x(\FF_x)$ and $E'_x(\FF_x)$  both have one point with $\defi{x}$-coordinate $a$.    Remembering the points at infinite, we find that $|E_x(\FF_x)| + |E'_x(\FF_x)| = 2q^{\deg x} +2$ and hence $a_x(E')=-a_x(E)$.
\end{proof}

Now suppose that $\det(I+A) \neq 0$.  By  Lemma~\ref{L:Zassenhaus} and (\ref{E:g and L}), we have 
 \[
 \spin(A) = 2^{N_E} \det(I+A) \cdot (\FF_\ell^\times)^2 = 2^{N_E} L(-1/q, E) \cdot (\FF_\ell^\times)^2.
\]  
Let $E'$ be an elliptic curve over $K$ that is a quadratic twist of $E$ by a non-square in $k^\times$.   Lemma~\ref{L:technical gamma}(\ref{L:technical gamma b}) implies that $L(T,E')=L(-T,E)$ and hence $L(1/q,E')=L(-1/q,E)$.  By Corollary~\ref{C:BSD main}, we deduce that $L(-1/q,E) \in  q^{g-1+\chi_{E'}}  c_{E'}  (\QQ^\times)^2$.   We have $\chi_{E'}=\chi_E$ since $E$ and $E'$ have the same Kodaira symbols by Lemma~\ref{L:technical gamma}(\ref{L:technical gamma a}).  By Lemma~\ref{L:technical gamma}(\ref{L:technical gamma c}), the integer $c_E c_{E'} \gamma_E = \prod_{x\in |C|} c_x(E)c_x(E') \gamma_x(E)^{\deg x}$ is a square.  Therefore,  $L(-1/q,E) \in  q^{g-1+\chi_{E}}  c_{E} \gamma_E  (\QQ^\times)^2$.  Since $\ell \nmid q^{g-1+\chi_{E}}  c_{E} \gamma_E $, we conclude that
\[
\spin(A) = 2^{N_E} q^{g-1+\chi_E} c_E \gamma_E \cdot (\FF_\ell^\times)^2.
\]
This completes the proof of (\ref{P:theta comp ii}).   Finally, suppose that $\det(I\pm A)\neq 0$.  By (\ref{P:theta comp ii}) and (\ref{P:theta comp iii}), we have 
\[
\spin(-I) = \spin(-A)\spin(A) = (2^{N_E} q^{g-1+\chi_E} c_E)\cdot (2^{N_E} q^{g-1+\chi_E} c_E \gamma_E) \cdot (\FF_\ell^\times)^2 = \gamma_E \cdot (\FF_\ell^\times)^2. 
\]
This proves part (\ref{P:theta comp iv}) since $\disc(V_{E,\ell})=\spin(-I)$.

\section{Families of quadratic twists}  \label{S:twists}

\subsection{Setup} \label{SS:twist setup}

Let $R$ be either a finite field whose characteristic is greater than $3$ or a ring of the form $\ZZ[S^{-1}]$ with $S$ a finite set of primes containing $2$ and $3$.  Fix a Weierstrass equation
\begin{equation} \label{E:initial WE}
\yy^2= \xx^3 + a_2(t) \xx^2 +a_4(t) \xx +a_6(t)
\end{equation}
with $a_i \in R[t]$ such that its discriminant $\Delta \in R[t]$ is non-zero.    Assume that the $j$-invariant $J(t)$ of the elliptic curve over $F(t)$ defined by (\ref{E:initial WE}), where $F$ is the quotient field of $R$, has non-constant $j$-invariant.   When $R=\ZZ[S^{-1}]$, we will allow ourselves to repeatedly enlarge the finite set $S$ so that various properties hold.   For example if $R$ has characteristic $0$, we shall assume that $\Delta(t)\not\equiv 0 \pmod{p}$ for all primes $p\notin S$.    
\\

We now consider quadratic twists by degree $1$ polynomials. Define the $R$-scheme 
\[
M=\Spec R[u,\Delta(u)^{-1}].
\]   
Let $k$ be any finite field that is an $R$-algebra (i.e., a finite extension of the field $R$ or a finite field whose characteristic does not lie in $S$).  Take any $m\in M(k)$, i.e., an element $m\in k$ with $\Delta(m)\neq 0$, and let $E_m$ be the elliptic curve over $k(t)$ defined by the Weierstrass equation
\begin{equation} \label{E:initial WE twist}
(t-m) \yy^2= \xx^3 + a_2(t) \xx^2 +a_4(t) \xx +a_6(t).
\end{equation}
We will prove the following in \S\ref{SS:kodaira independence}

\begin{lemma} \label{L:independence of p}
After possibly increasing the finite set $S$ when $R=\ZZ[S^{-1}]$, the multiset $\Kod(E_m)$ and the Kodaira symbol of $E_m$ at $\infty$ are independent of the choice of $k$ and $m$.
\end{lemma}

After possibly increasing the set $S$ when $R$ has characteristic $0$, we shall assume that the conclusions of Lemma~\ref{L:independence of p} hold.   Let $\Phi$ be the common multiset of Kodaira symbols from Lemma~\ref{L:independence of p}; the assumption that $J(t)$ is non-constant ensures that $\Phi$ is non-empty.    Let $\kappa_\infty$ be the common Kodaira symbol at $\infty$ of Lemma~\ref{L:independence of p}.   

The integers $N_{E_m}$, $\chi_{E_m}$, $\calL_{E_m}$ and $\gamma_{E_m}$ can be determined directly from $\Kod(E_m)=\Phi$, so they are independent of of $k$ and $m$; denote their common values by $N$, $\chi$, $\calL$ and $\gamma$, respectively.      

Define the integer $B_{E_m}:=\sum_{x\neq \infty} b_x(E_m)\deg x$, where the sum is over the closed points of $\Aff^1_k=\Spec k[t]$ and the  $b_x(E_m)$ are defined in \S\ref{SS:some invariants}.   Since $B_{E_m}$ can be determined directly from $\Phi$ and $\kappa_\infty$, we find that it is independent of $k$ and $m$; denote this common integer by $B$.
\\

Finally, take any prime $\ell\nmid 6 \calL$ that is not the characteristic of $R$.  If $R$ has characteristic $0$, we replace $S$ by $S\cup\{\ell\}$.

\subsection{Main representation} \label{SS:main rep}
Fix notation and assumptions as in \S\ref{SS:twist setup}.   The goal of this section is to prove the following proposition which give a representation of the \'etale fundamental group of $M$ that encodes the $L$-functions of the various quadratic twists $E_m$.

\begin{proposition} \label{P:main construction}
After possibly replacing $S$ by a larger finite set of primes when $R$ has characteristic $0$, there is an $N$-dimensional orthogonal space $V_\ell$  over $\FF_\ell$  and a continuous representation
\[
\theta_\ell \colon \pi_1(M) \to \Or(V_\ell)
\]
such that for any $R$-algebra $k$ that is a finite field of order $q$ and any $m\in M(k)$, the following hold:
\begin{alphenum}
\item \label{P:main construction a}
$\det(I-\theta_\ell(\Frob_m) T) \equiv L(T/q,E_m) \pmod{\ell}$,
\item \label{P:main construction b}
$\det(\theta_{\ell}(\Frob_m))=(-1)^{N}  \varepsilon_{E_m}$,
\item  \label{P:main construction c}
if $\det(I-\theta_{\ell}(\Frob_m))\neq 0$, then $\spin(-\theta_{\ell}(\Frob_m)) = 2^{N} q^{-1+\chi} \,c_{E_m} \cdot (\FF_\ell^\times)^2$,
\item \label{P:main construction d}
if $\det(I+\theta_{\ell}(\Frob_m))\neq 0$, then $\spin(\theta_{\ell}(\Frob_m)) = 2^{N} q^{-1+\chi} \,c_{E_m} \gamma \cdot (\FF_\ell^\times)^2$,
\item \label{P:main construction e}
if $\det(I\pm \theta_{\ell}(\Frob_m))\neq 0$, then $\disc(V_{\ell}) = \gamma \cdot (\FF_\ell^\times)^2$.
\end{alphenum}
\end{proposition}    

We now construct a lisse sheaf of $\FF_\ell$-modules that will give rise to the representation $\theta_\ell$ of Proposition~\ref{P:main construction}.  We have already defined the $R$-scheme $M=\Spec A$, where $A:=R[u, \Delta(u)^{-1}]$.  Set $C=\PP^1_{M}$; it is a smooth proper curve of genus $0$ over $M$ that can be obtained by extending $\Aff^1_M=\Spec A[t]$.    Define $U = \Spec A[t, (t-u)^{-1}, \Delta(t)^{-1}]$; it is an open $M$-subscheme of $C$.  After possibly enlarging the set $S$, we may  assume that the closed subscheme $D:=C-U$ of $C$ is \'etale over $M$.      

The Weierstrass equation
\[
(t-u) \yy^2= \xx^3 + a_2(t) \xx^2 +a_4(t) \xx +a_6(t).
\]
defines an elliptic curve $E\to U$.  Let $\calF:=E[\ell]$ be the $\ell$-torsion subscheme of $E$.  The morphism $\calF \hookrightarrow E \to U$ allows us to view $\calF$ as a lisse $\FF_\ell$-sheaf on $U$.   

Define the sheaf 
\[
\calG := R^1 \pi_*( j_*(\calF))
\] 
of $\FF_\ell$-modules on $M$, where $j\colon U \to C$ is the inclusion morphism and $\pi\colon C\to M$ is the structure morphism.  The Weil pairing gives an alternating pairing $\calF\times\calF \to \FF_\ell(1)$.  The cup product and this pairing on $\calF$ gives a symmetric pairing
\[
\calG \times \calG \to R^2 \pi_*(j_*(\calF)\otimes j_*(\calF)) \to R^2 \pi_*(j_*(\FF_\ell(1)))=\FF_\ell
\]
of sheaves on $M$.

\begin{lemma}
The sheaf $\calG$ is lisse.
\end{lemma}
\begin{proof}
Define $\bbar{\pi}=\pi\circ j$; it is the structure morphism $U\to M$.   We can identify $\calG$ with a subsheaf of  $R^1 \bbar{\pi}_* (\calF)$; for example by using the low degree terms of the Leray spectral sequence.    The homomorphism $R^1\bbar{\pi}_! (\calF) := R^1 \pi_*(j_!(\calF))\to \calG$ induced by the inclusion $j_!(\calF)\subseteq j_*(\calF)$ is surjective (this uses that $D\to M$ has relative dimension $0$).   Therefore, $\calG$ is the image of a homomorphism $R^1\bbar{\pi}_! (\calF)  \to R^1\bbar{\pi}_* (\calF)$.   It thus suffices to prove that the sheaves $R^1 \pi_!(\calF)$ and $R^1 \pi_*(\calF)$ are lisse.   Using Poincar\'e duality, it suffices to prove that $R^1 \pi_!(\calF)$ and $R^1 \pi_!(\calF^\vee)$ are lisse.

The sheaves $R^1 \pi_! (\calF)$ and $R^1 \pi_! (\calF^\vee)$ of $\FF_\ell$-modules are lisse by Corollaire~2.1.1 of \cite{MR629128}; the function $\varphi$ of loc.~cit.~ is constant since $D\to M$ is \'etale and the sheaves $\calF$ and $\calF^\vee$ are tamely ramified (since $2$ and $3$ are invertible in $A$).
\end{proof}

Now take any $m\in M(k)$, where $k$ is a finite field of order $q$ that is an $R$-algebra.  Base changing by $m$, we obtain from $E\to U \to M$ an open subvariety $U_m$ of $C_m =\PP^1_k$ and an elliptic curve $E_m\to U_m$.   The generic fiber of $E_m\to U_m$ is an elliptic curve defined over $k(t)$ given by the equation (\ref{E:initial WE twist}) with $u$ substituted by $m$ which, by abuse of notation, we have already denoted by $E_m$.   

Let $\bbar m$ be a geometric point of $M$ lying over $m$ obtained from an algebraic closure $\bbar{k}$ of $k$.   Let $\calG_{\bbar{m}}$ be the fiber of $\calG$ at $\bbar{m}$; it is an $\FF_\ell$-vector space that comes with a symmetric pairing $\ang{\,}{\,}$ from specializing the pairing on $\calG$.     The geometric Frobenius $\Frob_m$ acts on $\calG_{\bbar{m}}$.   By proper base change,  we have
\begin{equation} \label{E:specialization of G at m}
\calG_{\bbar{m}}=H^1(C_{\bbar{m}}, j'_*(E_m[\ell]))=H^1(C_{\bbar{k}},j'_*(E_m[\ell])) = V_{E_m,\ell},
\end{equation}
where $j'\colon U_m \hookrightarrow C_m$ is the inclusion morphism; the last equality uses Remark~\ref{R:alternate Vs}.   The induced pairing on $\calG_{\bbar{m}}$ agrees with the pairing on $V_{E_m,\ell}$ from \S\ref{SS:Lfunctions mod ell}.    With respect to (\ref{E:specialization of G at m}), the action of $\Frob_m$ on $\calG_{\bbar{m}}$ corresponds to the action of $\Frob_q$ on $V_{E,\ell}$.
\\

Let $\bbar{\xi}$ be a geometric generic point of $M$.   Our pairing on $\calG$ is non-degenerate since $\calG$ is lisse and it is non-degenerate on the fiber $\calG_{\bbar{m}}$.  Denote by $V_\ell$ the fiber of $\calG$ at $\bbar{\xi}$ with its pairing; it is an orthogonal space over $\FF_\ell$.   The lisse sheaf $\calG$ thus gives rise to a continuous representation
\[
\theta_\ell \colon \pi_1(M,\bbar{\xi}) \to \Or(V_\ell).
\]
With $m\in M(k)$ above, we find that there is an isomorphism $\varphi\colon V_\ell \xrightarrow{\sim} V_{E,\ell}$ of orthogonal spaces such that $\varphi^{-1} \circ \theta_{E,\ell}(\Frob_q) \circ \varphi$ lies in the same conjugacy class of $\Or(V_\ell)$ as $\theta_{\ell}(\Frob_m)$.  All the properties of $\theta_\ell$ given in Proposition~\ref{P:main construction} are now direct consequences of Propositions~\ref{P:main theta mod ell} and \ref{P:theta comp}.
 
\subsection{Big monodromy} \label{SS:big monodromy Hall}

Fix notation and assumptions as in \S\ref{SS:twist setup}.  After possibly increasing $S$ when $R$ has characteristic $0$, let
\[
\theta_\ell\colon \pi_1(M)\to\Or(V_\ell)
\]
be the representation of Proposition~\ref{P:main construction}.   

Let $\Phi'$ be the multiset consisting of $\Phi$ with one symbol $\kappa_\infty$ removed.  Assume further that the following conditions hold:
\begin{itemize}
\item
$\Phi'$ contains $\I_n$ for some $n\geq 1$,
\item
$\Phi'$ contains more than one $\I_0^*$, 
\item
$6B \leq N$.
\end{itemize}
The following explicit version of a theorem of Hall \cite{MR2372151} says that the image under $\theta_\ell$ of the geometric fundamental group is big.

\begin{theorem}[Hall] \label{T:main big monodromy}
With assumptions as above, the group $\theta_\ell(\pi_1(M_{\bbar{F}}))$ contains $\Omega(V_\ell)$ and is not a subgroup of $\SO(V_\ell)$, where $\bbar{F}$ is an algebraic closure of the quotient field $F$ of $R$.
\end{theorem}    

In \S\ref{SS:Hall sketch}, we will sketch some of the steps in Hall proof of Theorem~\ref{T:main big monodromy}.    The main reason for doing this is to ensure that all the conditions are explicit (in \cite{MR2372151}, one is allowed to replace the original curve by a suitably high degree twist so that the last two conditions before the statement of the theorem hold).  We will also need to refer to some of the details when handling the $n=5$ case of Theorem~\ref{T:Omega total}.

\subsection{Sketch of Theorem~\ref{T:main big monodromy}} \label{SS:Hall sketch}
First suppose that $R=\ZZ[S^{-1}]$ and let $\calG$ be the lisse sheaf on $M$ from \S\ref{SS:main rep}.   Take any $p\notin S$ and let $\theta_{p,\ell}\colon \pi_1(M_{\FF_p})\to \Or(V_\ell)$ be the representation obtained by specializing $\theta_\ell$, equivalently $\calG$, at the fiber of $M$ above $p$.   Since the formation of $\calG$ commutes with arbitrary base change, the representation $\theta_{p,\ell}$ agrees with the representation arising from the setup with \S\ref{SS:twist setup} by starting with the same Weierstrass equation except replacing $R$ by $\FF_p$.  For $p\notin S$ sufficiently large, all the conditions of Theorem~\ref{T:main big monodromy} hold.    Since $\theta_{p,\ell}(\pi_1(M_{\FFbar_p}))$ agrees with $\theta_\ell(\pi_1(M_{\Qbar}))$ for all sufficiently large $p$, it thus suffices to prove the theorem in the case where $R$ is a finite field.\\

Now assume that $R$ is a finite field $k$ whose characteristic is at least $5$.   We now describe the setup and key results of Hall from \S6 of \cite{MR2372151}.

Set $C=\PP^1_{k}$ and denote its function field by $K:=k(t)$.    Let $E_1$ be the elliptic curve over $K$ defined by (\ref{E:initial WE}).   For each non-zero polynomial $f\in k[t]$, let $E_f$ be the elliptic curve over $K$ obtained by taking the quadratic twist of $E_1$ by $f$.    \emph{Warning:} we are following Hall's notation throughout \S\ref{SS:Hall sketch}; the curve $E_{t-m}$ is denoted elsewhere in the paper by $E_m$.

We have $\ell\geq 5$ and $\ell$ is invertible in $k$.    The $j$-invariant $j_{E_1}\in K$ of $E_1$ is the same as the $j$-invariant of each $E_f$.  Therefore, $\calL_{E_f}$ is independent of $f$ and hence agrees with $\calL$.   In particular, our assumption $\ell\nmid 6\calL$ implies that $\ell$ does not divide $\max\{1,-v_x(j_{E_1})\}$ for all $x\in |C|$.    In \cite{MR2372151}*{\S6}, it is also assumed that $\ell$ is chosen so that the Galois group of the extension $K(E_1[\ell])/K$ contains a subgroup isomorphic to $\SL_2(\FF_\ell)$; however, this is a consequence of $\ell\nmid 6q\calL$ and Proposition~\ref{P:elliptic big monodromy}.

Let $\calE_f\to C$ be the N\'eron model of $E_f/K$ and let $\calE_{f}[\ell]$ be its $\ell$-torsion subscheme.  We have $V_{E_f,\ell}=H^1(C_{\bbar{k}}, \calE_{f}[\ell])$ by Remark~\ref{R:alternate Vs}
\\

For each integer $d \geq 0$, we let $F_d$ be the open subvariety of $\Aff_{k}^{d+1}$ consisting of tuples $(a_0,\ldots, a_d)$  for which the polynomial $\sum_{i=0}^d a_i t^i$ is separable of degree $d$ and relatively prime to $\Delta(t)\in k[t]$.   For each extension $k'/k$, we will identify each point $f\in F_d(k')$ with the corresponding degree $d$ polynomial in $k'[t]$.

Now assume that $d\geq 1$.  As noted in \cite{MR2372151}, there is an orthogonally self-dual lisse sheaf $\calT_{d,\ell}\to F_d$ of $\FF_\ell$-modules such that for any finite extension $k'\subseteq \bbar{k}$ of $k$ and any $f \in F_d(k')$, the (geometric) fiber of $\calT_{d,\ell}$ above $f$ is $H^1(C_{\bbar{k}}, \calE_{f}[\ell])=V_{E_f,\ell}$.  Moreover, the pairing on $\calT_{d,\ell}$ agrees with the pairing from \S\ref{SS:Lfunctions mod ell} on the fibers $V_{E_f,\ell}$.

Fix a polynomial $g\in F_{d-1}(k)$.  Let $U_g$ be the open subvariety of $\Aff^1_{k}$ consisting of $c$ for which $\Delta(c)g(c)\neq 0$.   We view $U_g$ as a closed subvariety of $F_d$ via the closed embedding $\varphi\colon U_g\hookrightarrow F_d$, $c\mapsto (c-t) g(t)$.   We then have an orthogonally self-dual lisse sheaf $\varphi^* (\calT_{d,\ell})$ of $\FF_\ell$-modules on $U_g$.   

In \S6.3 of \cite{MR2372151}, it is noted that $\varphi^*(\calT_{d,\ell})$ over $U_{g,\bbar{k}}$ agrees with the middle convolution sheaf $\MC_{-1}(\calE_g[\ell])$; this has the consequence that the sheaf $\varphi^*(\calT_{d,\ell})$ is geometrically irreducible and tame.\\

We now focus on the case with $d=1$ and $g=-1$.    The variety $U_g$ in $\Aff^1_{k}=\Spec k[u]$ is equal to $M=\Spec k[u,\Delta(u)^{-1}]$.   For each finite extension $k'/k$ and $m\in M(k')$, the (geometric) fiber of $\varphi^*(\calT_{d,\ell})$ above $m$ is $H^1(C_{\bbar{k}}, \calE_{t-m}[\ell])=V_{E_{t-m},\ell}$ (which is $V_{E_m,\ell}$ in the notation of \S\ref{SS:twist setup}).   

We find that the sheaf $\varphi^*(\calT_{d,\ell})$ over $M=U_g$ is precisely our sheaf $\calG$ from \S\ref{SS:main rep} and they have the same pairing.  We record the follow consequence for $\theta_\ell$.

\begin{lemma} \label{L:tame and irreducible}
The representations $\theta_\ell\colon \pi_1(M) \to \Or(V_\ell)$ is geometrically irreducible and tame. \qed
\end{lemma}

Since $\theta_\ell$ is tamely ramified, its restriction to $\pi_1(M_{\bbar{k}})$ factors through the maximal tame quotient $\pi^t_1(M_{\bbar{k}})$ of $\pi_1(M_{\bbar{k}})$.  Let $Z$ be the set of $\bbar{k}$-points of $\Aff^1_{k}-M$; it consists of the $c\in \bbar{k}$ for which $\Delta(c)=0$.   For each point $c\in Z \cup \{\infty\}$, let $\sigma_c$ be a generator of an inertia subgroup $\pi^t_1(M_{\bbar{k}})$ at $c$.  Choosing an ordering of the points $Z\cup\{\infty\}$, we may assume that the $\sigma_c$ are taken so that the product of the $\sigma_c$, with respect to the ordering, is trivial.    

The group $\pi^t_1(M_{\bbar{k}})$ is (topological) generated by $\{\sigma_c: c \in Z\}$; we do not need $\sigma_\infty$ since the product of the $\sigma_c$ is trivial.   In particular, $\{\theta_\ell(\sigma_c): c \in Z \}$ generates the group $\theta_\ell(\pi_1(M_{\bbar{k}}))$.

We need two quick group theory definitions.   For each $A\in \Or(V_\ell)$, we define $\Drop(A)$ to be the codimension in $V_\ell$ of the subspace fixed by $A$.   We say that an element $A\in \Or(V_\ell)$ is an \defi{isotropic shear} if it is non-trivial, unipotent and satisfies $(A-I)^2=0$.  

\begin{lemma} \label{L:G gen}
Fix a point $c \in Z$ and let $\kappa$ be the Kodaira symbol of $E_1/K$ at $t=c$.
\begin{romanenum}
\item \label{L:monodromy c 0} 
If $\kappa=\I_0$, then $\theta_{\ell}(\sigma_c)=I$.
\item \label{L:monodromy c a} 
If $\kappa=\I_n$ for some $n\geq 1$, then $\theta_{\ell}(\sigma_c)$ is a reflection.
\item \label{L:monodromy c b}
If $\kappa=\I_0^*$, then $\theta_{\ell}(\sigma_c)$ is an isotropic shear.
\item \label{L:monodromy c c}
We have $\Drop(\theta_{\ell}(\sigma_c))\leq 2$.
\end{romanenum}
\end{lemma}
\begin{proof}
The is a consequence of Lemma~6.5 of \cite{MR2372151} and also its proof for part (\ref{L:monodromy c c}). It is actually stated for $E_g=E_{-1}$ in loc.~cit., but $E_1$ and $E_{-1}$ have the same Kodaira symbols).
\end{proof}

\begin{remark} \label{R:foreshadow}
For later, we note that up to this point we have not made use of the three additional assumptions from \S\ref{SS:big monodromy Hall}.
\end{remark}

The following group theoretic result is a special case of Theorem~3.1 of \cite{MR2372151} with $r=2$.  

\begin{proposition} \label{P:monodromy criterion}
Let $G$ be an irreducible subgroup of $\Or(V_{\ell})$ generated by a set $\calS$.   Assume that $G$ contains a reflection and an isotropic shear.  Suppose that there is a subset $\calS_0\subseteq \calS$ satisfying the following properties:
\begin{alphenum}
\item
$\Drop(A) \leq 2$ for every $A\in \calS$,
\item
every $A\in \calS-\calS_0$ has order relative prime to $6$ or is a reflection,
\item
$6|\calS_0| \leq \dim_{\FF_\ell} V_\ell$.
\end{alphenum}
Then $G$ contains $\Omega(V_{\ell})$ and is not a subgroup of $\SO(V_{\ell})$.
\end{proposition}

We can now finish our sketch of Theorem~\ref{T:main big monodromy}.   Define the group $G:=\theta_\ell(\pi_1(M_{\bbar{k}}))$; it is irreducible since $\theta_\ell$ is geometrically irreducible.  The group $G$ is generated by the set $\calS:=\{\theta_\ell(\sigma_c): c \in Z\}$.   Let $\calS_0$ be the set of $\theta_\ell(\sigma_c)$ with $c\in Z$ for which $E_1$ has additive reduction at $t=c$ and the Kodaira symbol at $t=c$ is not $\I_0^*$. 

We may assume that $M(k)$ is non-empty (we may always replace $k$ by a finite extension at the beginning).  Fix any $m\in M(k)$.    The curve $E_{t-m}/K$ is a quadratic twist of $E_1/K$ by $t-m$.  Hence $E_{t-m}$ and $E_1$ have the same Kodaira symbol at each closed point $x$ of $\Aff^1_{k}$ except for one (correspond to the point $t=m$) for which $E_{t-m}$ has Kodaira symbol $\I_0^*$ and $E_1$ has good reduction.

The two assumptions of \S\ref{SS:big monodromy Hall} on $\Phi'$ imply that $E_1$ has multiplicative reduction at some $c\in Z$ and Kodaira symbol $\I_0^*$ at some $c\in Z$.    From Propositions~\ref{L:G gen}(\ref{L:monodromy c a}) and (\ref{L:monodromy c b}), we deduce that $G$ contains a reflection and an isotropic shear.  

Proposition~\ref{L:G gen}(\ref{L:monodromy c c}) implies that $\Drop(A)\leq 2$ for all $A\in \calS$.  For every $A\in \calS-\calS_0$, Proposition~\ref{L:G gen} implies that either $A$ is a reflection or that the order of $A$ is a power of $\ell$ (which is relatively prime to $6$ since $\ell\geq 5$).

We have $|\calS_0| \leq \sum_{x\neq \infty} b_x(E_1) \deg x$, where the sum is over the closed points $x$ of $\Aff^1_{k}$.  By our comparison of the Kodaira symbols of $E_1$ and $E_{t-m}$, we deduce that $|\calS_0| \leq B_{E_m}=B$.  Our assumption $6B\leq N$ then implies that $6|\calS_0| \leq 6B \leq N=\dim_{\FF_\ell} V_\ell$.

The conditions of Proposition~\ref{P:monodromy criterion} have all been verified and thus $G=\theta_\ell(\pi_1(M_{\bbar{k}}))$ contains $\Omega(V_\ell)$ and is not a subgroup of $\SO(V_\ell)$.

\subsection{Proof of Lemma~\ref{L:independence of p}} \label{SS:kodaira independence}
Let $E'$ be the elliptic curve over $k(t)$ given by (\ref{E:initial WE}); denote its Kodaira symbol at $\infty$ by $\kappa'_\infty$.

 Take any closed point $x$ of $\PP^1_k$.  Let $\kappa$ and $\kappa'$ be the Kodaira symbols of $E_m$ and $E'$, respectively, at $x$.   If $t^{-1}$ and $t-m$ are not uniformizers for the local field $K_x$, then we will have $\kappa=\kappa'$.   If $t-m$ is a uniformizer for $K_x$, then $E'$ has good reduction at $x$ (since $\Delta(m)\neq 0$) and hence $\kappa=\I_0^*$.    Finally, if $t^{-1}$ (and hence also $(t-m)^{-1}$) is a uniformer for $K_x$, then the pair $\{\kappa,\kappa'\}$ is one of the following: $\{\I_n,\I_n^*\}$, $\{\II,\IV^*\}$, $\{\IV,\II^*\}$, $\{\III,\III^*\}$; in particular, $\kappa$ is determined by $\kappa'=\kappa_\infty'$. Therefore, $\Kod(E_m)$ and the Kodaira symbol of $E_m$ at $\infty$ is determined by $\Kod(E')$ and $\kappa_\infty'$.  Observe that the quantities $\Kod(E')$ and $\kappa_\infty'$ do not change if we replace $E'$ by its base extension to $k'(t)$ where $k'$ is any finite extension of $k$.   The lemma is now immediate when $R$ is a finite field.   
 
 When $R=\ZZ[S^{-1}]$, it suffices to show that for the elliptic curve $E'$ over $\FF_p(t)$ given by (\ref{E:initial WE}), the multiset $\Kod(E')$ and the Kodaira symbol of $E'$ at $\infty$ is independent of $p$ for all sufficiently large $p\notin S$.   The equation (\ref{E:initial WE}) defines an elliptic surface $\pi\colon X\to \PP^1_\QQ$, where $\PP^1_\QQ = \Spec \QQ[t] \cup \{\infty\}$.  We may assume that $X$ is geometrically smooth and projective, and that $\pi$ is relatively minimal.     The singular fibers of $\pi$ are, geometrically, projectively lines and the pattern in which they intersect determines the Kodaira symbol of the fiber.   Choosing models and increasing $S$, we obtain a morphism $\calX \to \PP^1_{R}$ of $R$-scheme.   For all primes $p\notin S$, after possibly increasing $S$, we find that the number of singular fibers of $\calX_{\FFbar_p} \to \PP^1_{\FFbar_p}$, their Kodaira symbols (counted with multiplicity), and the Kodaira symbol above $\infty$ agree with those of $X_{\Qbar}=\calX_{\Qbar}\to \PP^1_{\Qbar}$.   This gives the desired independence.

\section{A criterion} \label{S:criterion}

In this section, we give a criterion for various  simple groups $\POmega(V)$ to occur as the Galois group of a regular extension of $\QQ(t)$.   The goal is not to give the most general formulation possible, but simply one that covers almost all of our cases.   \\

Consider a Weierstrass equation
\begin{equation} \label{E:weierstrass}
\yy^2 = \xx^3 + a_2(t) \xx^2 + a_4(t) \xx + a_6(t)
\end{equation}
with $a_2, a_4, a_6 \in \ZZ[t]$.  Let $\Delta \in \ZZ[t]$ be the discriminant of (\ref{E:weierstrass}) and assume that it is non-zero.  Assume that the $j$-invariant $J(t)\in \QQ(t)$ of the elliptic curve over $\QQ(t)$ defined by (\ref{E:initial WE}) is non-constant.

For each prime $p\geq 5$, let $M(\FF_p)$ be the set of $m\in \FF_p$ for which $\Delta(m)\neq 0$.  Let $E_m$ be the elliptic curve over $\FF_p(t)$ defined by the Weierstrass equation
\[
(t-m)\cdot \yy^2 = \xx^3 + a_2(t) \xx^2 + a_4(t) \xx + a_6(t).
\]
Let $\PP^1_{\FF_p}$ be the smooth proper curve over $\FF_p$ obtained by adjoining to $\Aff^1_{\FF_p}:=\Spec \FF_p[t]$ a point $\infty$; it has function field $\FF_p(t)$.  

From \S\ref{SS:twist setup}, we find that there are integers $N$, $\chi$, $\calL$ and $\gamma$ such $N=N_{E_m}$, $\chi=\chi_{E_m}$, $\calL=\calL_{E_m}$ and $\gamma=\gamma_{E_m}$ for all sufficiently large primes $p$ and all $m\in M(\FF_p)$.    From \S\ref{SS:twist setup}, there is an integer $B$ that equals $B_{E_m}:=\sum_{x\neq\infty} b_x(E_m) \deg x$ for all sufficiently large primes $p$ and all $m\in M(\FF_p)$, where the sum is over the closed points of $\Aff^1_{\FF_p}$.\\

Fix a non-constant $h \in \QQ(t)$ whose numerator and denominator have degree at most $4$.   There are unique relatively prime $\alpha, \beta \in \ZZ[t]$ such that the leading coefficient of $\beta$ is positive and $h=\alpha/\beta$.   For each prime $p\geq 5$, let $W(\FF_p)$ be the set of $w \in \FF_p$ that satisfy $\beta(w)\neq 0$ and $\Delta(h(w)) \neq 0$.  We have a map $W(\FF_p)\to M(\FF_p),$  $w\mapsto h(w)=\alpha(w)/\beta(w)$.  

Assume that the following hold for all for all sufficiently large primes $p$ and all $w\in W(\FF_p)$:
\begin{itemize}
\item
$E_{h(w)}$ has multiplicative reduction at some closed point of $\Aff^1_{\FF_p}$,
\item
$E_{h(w)}$ has Kodaira symbol $\I_0^*$ at more that one closed point of $\Aff^1_{\FF_p}$.
\end{itemize}
Assume also that $6B \leq N$.\\

Fix a prime $\ell \geq 5$ that does not divide $\calL$.   Assume that {one} of the following three conditions holds:
\begin{enumerate}[(A)]
\item \label{A} 
The integers $N$ and $\chi$ are odd, and $\gamma$ is a square modulo $\ell$.  For all sufficiently large primes $p$ and all $w\in W(\FF_p)$, the integer $2\cdot c_{E_{h(w)}}$ is a square modulo $\ell$.
\item \label{B}  
The integer $N$ is even and $\gamma$ is a non-square modulo $\ell$.  For all sufficiently large primes $p$ and all $w\in W(\FF_p)$, we have $\varepsilon_{E_{h(w)}}=1$.
 \item \label{C}   
The integer $N$ is even, $\chi$ is odd, and $\gamma$ is a square modulo $\ell$.   For all sufficiently large primes $p$ and all $w\in W(\FF_p)$, we have $\varepsilon_{E_{h(w)}}=1$ and the integer $c_{E_{h(w)}}$ is a square modulo $\ell$.   
\end{enumerate}   

\begin{theorem} \label{T:main criterion}
Fix notation and assumptions as above.    Let $V$ be an orthogonal space of dimension $N$ over $\FF_\ell$.   If $N$ is even, suppose further that $\disc(V)= \gamma\cdot  (\FF_\ell^\times)^2$.   Then the group $\Omega(V)$, and hence also $\POmega(V)$, occurs as the Galois group of a regular extension of $\QQ(t)$.
\end{theorem}
\begin{proof}
For a finite set of primes $S$, define the ring $R=\ZZ[S^{-1}]$.  We will allow ourselves to increase the finite set $S$ to ensure various conditions hold; for example, we will assume that $S$ contains $2$, $3$, $\ell$ and the primes $p$ for which $\Delta(t)\equiv 0 \pmod{p}$.   We may also assume that all the conditions that are assumed to hold for sufficiently large primes $p$ actually hold for all $p\notin S$.

Define the $R$-scheme $M=\Spec R[u,\Delta(u)^{-1}]$.  For each prime $p\notin S$, $M(\FF_p)$ is indeed the set of $m\in \FF_p$ for which $\Delta(m)\neq 0$.      After possibly increasing $S$, there is an orthogonal space $V_\ell$ over $\FF_\ell$ of dimension $N$ and a representation
\[
\theta_\ell\colon \pi_1(M)\to \Or(V_\ell)
\]
satisfying the conclusions of Proposition~\ref{P:main construction}. 

Let $\Phi'$ be the multiset of Kodaira symbols as in \S\ref{SS:big monodromy Hall}.   Our assumption that $E_m$ has multiplicative reduction at some closed point $\Aff^1_{\FF_p}$ implies that $\Phi'$ contains a symbol $\I_n$ for some $n\geq 1$.   Our assumption that $E_m$ has at least Kodaira symbol $\I_0^*$ at more than one closed point $\Aff^1_{\FF_p}$ implies that $\Phi'$ contains the symbol $\I_0^*$ at least twice.   We have $6B\leq N$ by assumption.   Theorem~\ref{T:main big monodromy} now applies and thus $\theta_\ell(\pi_1(M_{\Qbar}))\supseteq \Omega(V_\ell)$.
\\

Define the $R$-scheme $W=\Spec R[v, \beta(v)^{-1}, \Delta(h(v))^{-1}]$.  For each prime $p\notin S$, $W(\FF_p)$ is indeed the set of $w\in \FF_p$ for which $\beta(w)\neq0$ and $\Delta(h(w))\neq 0$.  Define the morphism
\[
\varphi\colon W \to M,\, w \mapsto h(w).
\]
We can replace $a_i(t)$ by $a_i(t) f(t)^i$ for a fixed non-zero separable polynomial $f(t)\in \ZZ[t]$ that is relatively prime to $\Delta(t)$; the new discriminant equals $f(t)^6\Delta(t)$ and all the assumptions of the theorem still hold.    We may choose $f$ so that the morphism $W_\QQ \to M_\QQ$ is finite \'etale.   After possible increasing the set $S$, we may thus assume that $\varphi$ is also finite \'etale.  

The morphism $\varphi$ thus gives rise to an injective homomorphism $\varphi_* \colon \pi_1(W) \hookrightarrow \pi_1(M)$; uniquely determined up to conjugacy.  Let 
\[
\vartheta_\ell \colon \pi_1(W)\to \Or(V_\ell)
\]
be the representation obtained by composing $\varphi_*$ and $\theta_\ell$.    For each prime $p\notin S$ and $w\in W(\FF_p)$, we have an equality 
\[
\vartheta_\ell(\Frob_w)= \theta_\ell(\Frob_{h(w)})
\]
of conjugacy classes in $\Or(V_\ell)$.  Define the groups
\[
G:=\vartheta_\ell(\pi_1(W)) \quad\text{ and }\quad G^g:=\vartheta_\ell(\pi_1(W_\Qbar)).
\]  

We claim that $G^g \supseteq \Omega(V_\ell)$.  The \'etale morphism $\varphi$ has degree at most $4$ by our assumption on the degree of the numerator and denominator of $h(t)$.  Since $\mathfrak{S}_4$ is solvable, there is a normal open subgroup $H$ of $\pi_1(M_{\Qbar})$ such that $H\subseteq \varphi_*(\pi_1(W_{\Qbar}))$ and such that $\pi_1(M_{\Qbar})/H$ is solvable.    The group $\Omega(V_\ell)$ is perfect and non-abelian and we have seen that it is a normal subgroup of $\theta_\ell(\pi_1(M_{\Qbar}))$.  Therefore, $\theta_\ell(H)$ contains $\Omega(V_\ell)$ since the quotient $\pi_1(M_{\Qbar})/H$ is solvable.   This proves the claim since $G^g\supseteq \theta_\ell(H)$.

We will now show that $G$ is a subgroup of $\pm \Omega(V_\ell)$.\\

Let us first assume that $N$ is odd and hence assumption~(\ref{A}) holds.    Let $\kappa$ be any coset of $\Omega(V_\ell)$ in $G$.  Take $e\in \{0,1\}$ such that $\det(\kappa)=\{ (-1)^e \}$.  There exists an element $A\in \kappa$ such that $\det(I+(-1)^e A)\neq 0$.  We have $\det((-I)^e A) = (-1)^e (-1)^e =1$ since $N$ is odd.  Using equidistribution, there is a prime $p\notin S$ and an element $w\in W(\FF_p)$ such that the conjugacy class $\vartheta_\ell(\Frob_w)=\theta_\ell(\Frob_{h(w)})$ of $\Or(V_\ell)$ contains $A$.  

By assumption~(\ref{A}), $\gamma$ is a square modulo $\ell$ (it is non-zero modulo $\ell$ since $\ell\nmid 6\calL$).  Therefore, $\spin((-1)^e A) = \spin((-1)^e \theta_\ell(\Frob_{h(w)}))$ equals $2^N q^{-1+\chi} c_{E_{h(w)}}  (\FF_\ell^\times)^2=2 q^{-1+\chi} c_{E_{h(w)}}  (\FF_\ell^\times)^2$ by parts (\ref{P:main construction c}) and (\ref{P:main construction d}) of Proposition~\ref{P:main construction}.    By assumption~(\ref{A}), $\chi$ is odd and $2\cdot c_{E_{h(w)}}$ is a square modulo $\ell$.   Therefore, $\spin((-1)^e A) = (\FF_\ell^\times)^2$.

Since $(-1)^e A$ has trivial determinant and spinor norm, it belongs to $\Omega(V_\ell)$.   The coset $\kappa=A\Omega(V_\ell)$ is thus either $\Omega(V_\ell)$ or $-\Omega(V_\ell)$.   Therefore, $G \subseteq \pm \Omega(V_\ell)$ since $\kappa$ was an arbitrary coset of $\Omega(V_\ell)$ in $G$.
\\

Now suppose that $N$ is even and hence assumption~(\ref{B}) or (\ref{C}) holds.   Since $G\supseteq \Omega(V_\ell)$ and $N$ is even, there is an element $A\in G$ such that $\det(I\pm A) \neq 0$.   Using equidistribution, there is a prime $p\notin S$ and an element $w\in W(\FF_p)$ such that the conjugacy class $\vartheta_\ell(\Frob_w)=\theta_\ell(\Frob_{h(w)})$ of $\Or(V_\ell)$ contains $A$.    By part (\ref{P:main construction e}) of Proposition~\ref{P:main construction}, we have $\disc(V_\ell) = \gamma (\FF_\ell^\times)^2$.

  By Proposition~\ref{P:main construction}(\ref{P:main construction b}), we find that 
\[
\det(\vartheta_\ell(\Frob_w)) = \det( \theta_\ell(\Frob_{h(w)}) ) = (-1)^N  \varepsilon_{E_{h(w)}} = \varepsilon_{E_{h(w)}}
\]
for all sufficiently large primes $p \notin S$ and all $w\in W(\FF_p)$.   Assumption~(\ref{B}) or (\ref{C}) then implies that $\det(\vartheta_\ell(\Frob_w))=1$ for all large $p$ and all $w\in W(\FF_p)$.  Using equidistribution, we deduce that $G$ is a subgroup of $\SO(V_\ell)$.

Suppose that $\gamma$ is not a square modulo $\ell$.   Then $-I\notin \Omega(V_\ell)$ since $\spin(-I)=\disc(V_\ell) = \gamma(\FF_\ell^\times)^2$.   Therefore, $G$ is a subgroup of $\SO(V_\ell)$ and $\SO(V_\ell)=\pm \Omega(V_\ell)$.

Now assume that $\gamma$ is a square modulo $\ell$ and hence that assumption~(\ref{C}) holds.   Let $\kappa$ be any coset of $\Omega(V_\ell)$ in $G$.  Since $G\subseteq \SO(V_\ell)$, there is an element $A\in \kappa$ such that $\det(I + A)\neq 0$.  Using equidistribution, there is a prime $p\notin S$ and an element $w\in W(\FF_p)$ such that the conjugacy class $\vartheta_\ell(\Frob_w)=\theta_\ell(\Frob_{h(w)})$ of $\Or(V_\ell)$ contains $A$.    By part (\ref{P:main construction d}) of Proposition~\ref{P:main construction}, $\spin(A)$ equals $2^N q^{-1+\chi} c_{E_{h(w)}}  (\FF_\ell^\times)^2=q^{-1+\chi} c_{E_{h(w)}} (\FF_\ell^\times)^2$.   By assumption~(\ref{C}), $\chi$ is odd and $ c_{E_{h(w)}}$ is a square modulo $\ell$.   Therefore, $\spin(A) = (\FF_\ell^\times)^2$.   So $\kappa=A\Omega(V_\ell)$ is $\Omega(V_\ell)$ and hence $G=\Omega(V_\ell)$ since $\kappa$ was an arbitrary coset.\\

We have proved the inclusions $\Omega(V_\ell) \subseteq G^g \subseteq G \subseteq \pm \Omega(V_\ell)$.   Let $Z$ be the group $\{I\}$ if $-I\in \Omega(V_\ell)$ and the group $\{\pm I\}$ if $-I \notin \Omega(V_\ell)$.    The natural map $\Omega(V_\ell)  \to (\pm \Omega(V_\ell))/Z$ is thus an isomorphism.    Since $G\subseteq \pm \Omega(V_\ell)$, we can define the homomorphism
\[
\beta\colon \pi_1(W) \xrightarrow{\vartheta_\ell} \pm \Omega(V_\ell) \twoheadrightarrow (\pm \Omega(V_\ell))/Z \cong \Omega(V_\ell).
\]
We have $\beta(\pi_1(W))=\beta(\pi_1(W_{\Qbar}))=\Omega(V_\ell)$ since $G^g\supseteq \Omega(V_\ell)$.    Therefore, $\beta$ gives rise to a regular extension of $\QQ(v)$, the function field of $W$, that is Galois with Galois group isomorphic to $\Omega(V_\ell)$.     This completes the proof of the theorem; when $N$ is even, we have already shown that $\disc(V_\ell) =\gamma (\FF_\ell^\times)^2$.
\end{proof}

\section{\texorpdfstring{Proof of Theorem~\ref{T:Omega total}($i$)  for $n>5$}{Proof of Theorem~\ref{T:Omega total}(i)  for n>5}} \label{SS:numerics odd}

Take any prime $\ell\geq 5$.  In this section, we use the setup of \S\ref{S:criterion} to show that $\Omega_N(\ell)$ occurs as the Galois group of a regular extension of $\QQ(t)$ for every odd integer $N>5$.  The case $N=5$ requires extra attention and will be discussed in \S\ref{S:N=5}.   The proof is broken up into four cases depending on the value of $N$ modulo $8$.

\subsection{\texorpdfstring{$N\equiv 1 \pmod{8}$}{N congruent to 1 mod 8}}
Take any integer $n\geq 1$.  Define the rational function $h(u)=u$ and the polynomial $f(t)=\prod_{i=1}^{4n}(t-(i+1))$.  Consider the Weierstrass equation
\begin{equation} \label{E:1}
\yy^2= \xx\cdot (\xx-f(t))\cdot (\xx-tf(t)) = \xx^3-(t+1)f(t) \xx^2 + t f(t)^2 \xx;
\end{equation}
it has discriminant $\Delta(t)=16 f(t)^6 t^2 (t-1)^2$ and the $j$-invariant of the corresponding elliptic curve over $\QQ(t)$ is $2^8(t^2-t+1)^3 t^{-2} (t-1)^{-2}$.   
  
Now take notation as in \S\ref{S:criterion}.  Take any prime $p\nmid 6\ell$ such that $f(t)$ modulo $p$ is separable and $f(0)f(1)\not\equiv 0 \pmod{p}$.   Take any $w\in W(\FF_p)$, i.e., any $w\in \FF_p$ for which $\Delta(w)\neq 0$.   Let $x$ be any closed point of $\PP^1_{\FF_p}=\Spec \FF_p[t] \cup \{\infty\}$ for which $E_{h(w)}/\FF_p(t)$ has bad reduction and let $\kappa_x$ be the Kodaira symbol of $E_{h(w)}/\FF_p(t)$ at $x$.

\begin{itemize}
\item
Suppose $x=0$.  We have $\kappa_x=\I_2$, so $c_x(E_{h(w)})=2$.
\item
Suppose $x=1$.  We have $\kappa_x=\I_2$, so $c_x(E_{h(w)})=2$.
\item
Suppose $x=\infty$.   We have $\kappa_x=\I_2$, so $c_x(E_{h(w)})=2$.
\item
Suppose $x=a$ is a root of $(t-h(w))f(t) \bmod{p} \in \FF_p[t]$.  We have $\kappa_x=\I_0^*$.    Using that the degree $3$ polynomial of $\xx$ in the Weierstrass equation (\ref{E:1}) factors into linear terms, Tate's algorithm shows that $c_x(E_{h(w)})=4$.  
\end{itemize}
Note that the curve $E_{h(w)}$ has multiplicative reduction at a closed point $x\neq\infty$ and Kodaira symbol $\I_0^*$ at more than one closed point $x\neq \infty$.   

From the computations above, we find that $N_{E_{h(w)}}=-4+ 3\cdot 1+ (4n+1)\cdot 2  =8n+1$, $\chi_{E_{h(w)}} = (2+2+2+6(4n+1))/12= 2n+1$, $\calL_{E_{h(w)}}=1$, $\gamma_{E_{h(w)}}=1$ and $B_{E_{h(w)}}=0$.  Since this hold for all large primes $p$ and all $w\in W(\FF_p)$, we thus have $N=8n+1$, $\chi=2n+1$, $\calL=1$, $\gamma=1$ and $B=0$.    Observe that $N$ is odd, $\chi$ is odd and $\gamma$ is a square modulo $\ell$.   We have $6B=0\leq 8n+1=N$.   

We have verified all the conditions of \S\ref{S:criterion} and in particular that assumption (\ref{A}) holds.  From Theorem~\ref{T:main criterion}, we deduce that $\Omega(V)$ occurs as the Galois group of a regular extension of $\QQ(t)$ where $V$ is an orthogonal space over $\FF_\ell$ of dimension $N=8n+1$ with $n\geq 1$.

\subsection{\texorpdfstring{$N\equiv 3 \pmod{8}$}{N congruent to 3 mod 8}}
Take any integer $n\geq 1$.    Define the rational function $h(u)=3u^2/(3u^2+1)$ and the polynomial $f(t)=\prod_{i=1}^{4n}(t-h(i))$.   Consider the Weierstrass equation
\[
\yy^2 = \xx^3-3tf(t)^2 \xx + 2t^2 f(t)^3;
\]
it has discriminant $\Delta(t)=-2^6 3^3 f(t)^6 t^3(t-1)$ and the $j$-invariant of the corresponding elliptic curve over $\QQ(t)$ is $-1728 (t-1)^{-1}$.   

Now take notation as in \S\ref{S:criterion}.   Take any prime $p\nmid 6\ell$ such that $f(t)$ modulo $p$ is separable and $f(0)f(1)\not\equiv 0 \pmod{p}$.   Take any $w\in W(\FF_p)$, i.e., any $w\in \FF_p$ for which $3w^2+1\neq 0$ and $\Delta(w)\neq 0$.   Let $x$ be any closed point of $\PP^1_{\FF_p}=\Spec \FF_p[t] \cup \{\infty\}$ for which $E_{h(w)}/\FF_p(t)$ has bad reduction and let $\kappa_x$ be the Kodaira symbol of $E_{h(w)}/\FF_p(t)$ at $x$.

\begin{itemize}
\item
Suppose $x=\infty$.   We have $\kappa_x=\II$, so $c_x(E_{h(w)})=1$.
\item
Suppose $x=0$.   We have $\kappa_x=\III$, so $c_x(E_{h(w)})=2$.
\item
Suppose $x=1$.   We have $\kappa_x=\I_1$, so $c_x(E_{h(w)})=1$.
\item
Suppose $x=a$ is a root of $(t-h(w))f(t) \bmod{p} \in \FF_p[t]$.  We have $\kappa_x=\I_0^*$.     Tate's algorithm shows that $c_x(E_{h(w)})=1+m$ where $m$ is the number of roots of 
\[
P(\xx):=\xx^3-3a \xx + 2a^2  
\] 
in $\FF_p$.  Using that $a=h(b)$ for some $b\in \FF_p$, we find that the discriminant of $P(\xx)$ is a non-zero square (moreover, it equals $54b^3/(3b^2 + 1)^2$ squared), so $m$ equals $0$ or $3$.  Therefore, $c_x(E_{h(w)})$ equals $1$ or $4$.
\end{itemize}
Note that the curve $E_{h(w)}$ has multiplicative reduction at a closed point $x\neq\infty$ and Kodaira symbol $\I_0^*$ at more than one closed point $x\neq \infty$.   

From the computations above, we find that $N_{E_{h(w)}}=-4+ 2+2+1+ (4n+1)\cdot 2  =8n+3$, $\chi_{E_{h(w)}} = (2+3+1+6(4n+1))/12= 2n+1$, $\calL_{E_{h(w)}}=1$, $\gamma_{E_{h(w)}}=1$ and $B_{E_{h(w)}}=1$.  Since this hold for all large primes $p$ and all $w\in W(\FF_p)$, we thus have $N=8n+3$, $\chi=2n+1$, $\calL=1$, $\gamma=1$ and $B=1$.    Observe that $N$ is odd, $\chi$ is odd and $\gamma$ is a square modulo $\ell$.   We have $6B=6\leq 8n+1=N$.   

We have verified all the conditions of \S\ref{S:criterion} and in particular that assumption (\ref{A}) holds.  From Theorem~\ref{T:main criterion}, we deduce that $\Omega(V)$ occurs as the Galois group of a regular extension of $\QQ(t)$ where $V$ is an orthogonal space over $\FF_\ell$ of dimension $N=8n+3$  with $n\geq 1$.

\subsection{\texorpdfstring{$N\equiv 5 \pmod{8}$ and $N>5$}{N>5 congruent to 5 mod 8}}
\label{SS:5 mod 8}
Take any integer $n\geq 1$.    Define the rational function $h(u)=(-u^2 + 3)/(u^2 + 3)$ and the polynomial $f(t)=\prod_{i=1}^{4n}(t-h(i))$.   Consider the Weierstrass equation
\[
\yy^2 = \xx^3+ 3(t^2-1)^3 f(t)^2 \xx -2(t^2-1)^5 f(t)^3;
\]
it has discriminant $\Delta(t)=-1728 f(t)^6 t^2 (t-1)^9(t+1)^9$ and the $j$-invariant of the corresponding elliptic curve over $\QQ(t)$ is $1728 t^{-2}$.  

Now take notation as in \S\ref{S:criterion}.   Take any prime $p\nmid 6\ell$ such that $f(t)$ modulo $p$ is separable and $f(0)f(1)f(-1)\not\equiv 0 \pmod{p}$.    Take any $w\in W(\FF_p)$, i.e., any $w\in \FF_p$ for which $w^2+3\neq 0$ and $\Delta(h(w))\neq 0$.   Let $x$ be any closed point of $\PP^1_{\FF_p}=\Spec \FF_p[t] \cup \{\infty\}$ for which $E_{h(w)}/\FF_p(t)$ has bad reduction and let $\kappa_x$ be the Kodaira symbol of $E_{h(w)}/\FF_p(t)$ at $x$.

\begin{itemize}
\item
Suppose $x=\infty$.  We have $\kappa_x=\II^*$, so $c_x(E_{h(w)})=1$.
\item
Suppose $x=0$.   We have $\kappa_x=\I_2$, so $c_x(E_{h(w)})=2$.
\item
Suppose $x=1$.   We have $\kappa_x=\III^*$, so $c_x(E_{h(w)})=2$.
\item
Suppose $x=-1$.   We have $\kappa_x=\III^*$, so $c_x(E_{h(w)})=2$.
\item
Suppose $x=a$ is a root of $(t-h(w))f(t) \bmod{p} \in \FF_p[t]$.   We have $\kappa_x=\I_0^*$.   Tate's algorithm shows that $c_x(E_{h(w)})=1+m$ where $m$ is the number of roots of 
\[
P(\xx):=\xx^3+ 3(a^2-1)^3 \xx -2(a^2-1)^5 
\] 
in $\FF_p$.  Using that $a=h(b)$ for some $b\in \FF_p$, we find that the discriminant of $P(\xx)$ is a non-zero square (moreover, it equals $2^{10} 3^6 b^9(b^2-3)/(b^2 + 3)^{10}$ squared), so $m$ equals $0$ or $3$.   Therefore, $c_x(E_{h(w)})$ equals $1$ or $4$.
\end{itemize}
Note that the curve $E_{h(w)}$ has multiplicative reduction at a closed point $x\neq\infty$ and Kodaira symbol $\I_0^*$ at more than one closed point $x\neq \infty$.   

From the computations above, we find that $N_{E_{h(w)}}=-4+ 2+1+2+2+ (4n+1)\cdot 2  =8n+5$, $\chi_{E_{h(w)}} = (10+2+9+9+6(4n+1))/12= 2n+3$, $\calL_{E_{h(w)}}=1$, $\gamma_{E_{h(w)}}=1$ and $B_{E_{h(w)}}=2$.  Since this hold for all large primes $p$ and all $w\in W(\FF_p)$, we thus have $N=8n+5$, $\chi=2n+3$, $\calL=1$, $\gamma=1$ and $B=2$.    Observe that $N$ is odd, $\chi$ is odd and $\gamma$ is a square modulo $\ell$.   We have $6B=12\leq 13\leq N$.   

We have verified all the conditions of \S\ref{S:criterion} and in particular that assumption (\ref{A}) holds.  From Theorem~\ref{T:main criterion}, we deduce that $\Omega(V)$ occurs as the Galois group of a regular extension of $\QQ(t)$ where $V$ is an orthogonal space over $\FF_\ell$ of dimension $N=8n+5$ with $n\geq 1$.

\begin{remark} \label{R:foreshadow computations}
For later, note that the above computations hold when $n=0$ except for two things:   The first is that each $E_{h(w)}$ has Kodaira symbol $\I^*_0$ at only one closed point $x\neq \infty$ of $\PP^1_{\FF_p}$.   The second is that $6B=12$ is now greater than $N=5$.
\end{remark}

\subsection{\texorpdfstring{$N\equiv 7 \pmod{8}$}{N congruent to 7 mod 8}}

Take any integer $n\geq 0$.  Define the rational function $h(u)=u$ and the polynomial $f(t)=\prod_{i=1}^{4n+2}(t-(i+1))$.  Consider the Weierstrass equation
\begin{equation} \label{E:7}
\yy^2= \xx\cdot (\xx+tf(t))\cdot (\xx+t^2f(t)) = \xx^3+t(t+1)f(t)\xx^2 + t^3 f(t)^2 \xx;
\end{equation}
it has discriminant $\Delta(t)=16 f(t)^6 t^8 (t-1)^2$ and the $j$-invariant of the corresponding elliptic curve over $\QQ(t)$ is $2^8(t^2-t+1)^3 t^{-2} (t-1)^{-2}$.   

Now take notation as in \S\ref{S:criterion}.  Take any prime $p\nmid 6\ell$ such that $f(t)$ modulo $p$ is separable and $f(0)f(1)\not\equiv 0 \pmod{p}$.  Take any $w\in W(\FF_p)$, i.e., any $w\in \FF_p$ for which $\Delta(h(w))\neq 0$.   Let $x$ be any closed point of $\PP^1_{\FF_p}=\Spec \FF_p[t] \cup \{\infty\}$ for which $E_{h(w)}/\FF_p(t)$ has bad reduction and let $\kappa_x$ be the Kodaira symbol of $E_{h(w)}/\FF_p(t)$ at $x$.

\begin{itemize}
\item
Suppose $x=\infty$.   Over $\FF_p(t)_x=\FF_p(\!(t^{-1})\!)$, the elliptic curve $E_{h(w)}$ is isomorphic to the curve defined by the Weierstrass equation
\[
\yy^2= \xx(\xx+t^{-2})(\xx+t^{-1}) =\xx^3+(1+t^{-1})\cdot t^{-1} \xx^2 + t^{-3}\xx.
\]
We have $\kappa_x=\I_2^*$. Using Tate's algorithm, we find that $c_x(E_{h(w)})=4$ (since the quadratic equation arising has vanishing constant term).
\item
Suppose $x=0$.  Over $\FF_p(t)_x$, the elliptic curve $E_{h(w)}$ is isomorphic to the curve defined by the Weierstrass equation
\[
\yy^2= \xx^3-h(w)f(0) (t+1)\cdot t\, \xx^2 + h(w)^2f(0)^2\cdot  t^3 \xx.
\]
We have $\kappa_x=\I_2^*$.  Using Tate's algorithm, we find that $c_x(E_{h(w)})=4$ (since the quadratic equation arising has vanishing constant term).
\item
Suppose $x=1$.  We have $\kappa_x=\I_2$, so $c_x(E_{h(w)})=2$.
\item
Suppose $x=a$ is a root of $(t-h(w))f(t) \bmod{p} \in \FF_p[t]$.  We have $\kappa_x=\I_0^*$.   We have $c_x(E_{h(w)})=4$ since the polynomial of $\xx$ in (\ref{E:7}) splits into linear factors.
\end{itemize}
Note that the curve $E_{h(w)}$ has multiplicative reduction at a closed point $x\neq\infty$ and Kodaira symbol $\I_0^*$ at more than one closed point $x\neq \infty$.   

From the computations above, we find that $N_{E_{h(w)}}=-4+ 2+2+1+ (4n+3)\cdot 2  =8n+7$, $\chi_{E_{h(w)}} = (8+8+2+6(4n+3))/12= 2n+3$, $\calL_{E_{h(w)}}=1$, $\gamma_{E_{h(w)}}=1$ and $B_{E_{h(w)}}=1$.  Since this hold for all large primes $p$ and all $w\in W(\FF_p)$, we thus have $N=8n+7$, $\chi=2n+3$, $\calL=1$, $\gamma=1$ and $B=1$.    Observe that $N$ is odd, $\chi$ is odd and $\gamma$ is a square modulo $\ell$.   We have $6B=6\leq 7\leq N$.   

We have verified all the conditions of \S\ref{S:criterion} and in particular that assumption (\ref{A}) holds.  From Theorem~\ref{T:main criterion}, we deduce that $\Omega(V)$ occurs as the Galois group of a regular extension of $\QQ(t)$ where $V$ is an orthogonal space over $\FF_\ell$ of dimension $N=8n+7$ with $n\geq 0$.

\section{\texorpdfstring{Proof of Theorem~\ref{T:Omega total}($ii$) and ($iii$)}{Proof of Theorem~\ref{T:Omega total}(ii) and (iii)}} \label{SS:numerics even 1}

Take any prime $\ell\geq 5$ and even integer $N\geq 6$.   Let $V$ be an orthogonal space of dimension $N$ over $\FF_\ell$ satisfying $\det(V)=(\FF_\ell^\times)^2$.  We use the criterion of \S\ref{S:criterion} to show that $\Omega(V)$, and hence $\POmega(V)$, occurs as the Galois group of a regular extension of $\QQ(t)$.  Note that $\POmega(V)$ is isomorphic to $\POmega^+_N(\ell)$ if $N\equiv 0 \pmod{4}$ or $\ell\equiv 1 \pmod{4}$ and isomorphic to $\POmega^-_N(\ell)$ if $N\equiv 2\pmod{4}$ and $\ell\equiv 3\pmod{4}$.

The proof is broken up into four cases depending on the value of $N$ modulo $8$.   

\subsection{\texorpdfstring{$N\equiv 0 \pmod{8}$}{N congruent to 0 mod 8}}
Take any integer $n\geq 1$.  Define the rational function $h(u)= 4 u^2/(u^2+1)^2$  and the polynomial  $f(t)=\prod_{i=1}^{4n-1}(t-h(i+1))$.  Consider the Weierstrass equation
\begin{equation} \label{E:00}
\yy^2 = \xx^3- 3 (t-1)^3 (t-4) f(t)^2 \xx -2(t-1)^5(t+8) f(t)^3;
\end{equation}
it has discriminant $\Delta(t)=-2^6 3^6 f(t)^6 t^2 (t-1)^9$ and the $j$-invariant of the corresponding elliptic curve over $\QQ(t)$ is $-64(t-4)^3 t^{-2}$.  

Now take notation as in \S\ref{S:criterion}.  Take any prime $p\nmid 6\ell$ such that $f(t)$ modulo $p$ is well-defined and separable, and $f(0)f(1)\not\equiv 0 \pmod{p}$.    Take any $w\in W(\FF_p)$, i.e., any $w\in \FF_p$ for which $w^2+1\neq0$ and $\Delta(h(w))\neq0$.  Let $x$ be any closed point of $\PP^1_{\FF_p}= \Spec \FF_p[t] \cup \{\infty\}$ for which $E_{h(w)}/\FF_p(t)$ has bad reduction and let $\kappa_x$ be the Kodaira symbol of $E_{h(w)}/\FF_p(t)$ at $x$.

\begin{itemize}
\item
Suppose $x=\infty$.  We have $\kappa_x=\I_1$ and hence $c_x(E_{h(w)})=1$.   Over $\FF_p(t)_x=\FF_p(\!(t^{-1})\!)$, the elliptic curve $E_{h(w)}$ is isomorphic to the curve defined by (\ref{E:00}) and hence also by $\yy^2=\xx^3-3(1-t^{-1})(1-4t^{-1}) \xx-2(1-t^{-1})^2(1+8t^{-1})$.  Reducing to $\FF_x$, we have the equation
\[
\yy^2= \xx^3-3\xx-2 = -3(\xx+1)^2+(\xx+1)^3.
\]
Therefore, the curve $E_{h(w)}$ has split multiplicative reduction at $x$ if and only if $-3$ is a square in $\FF_p$.
\item
Suppose $x=0$. We have $\kappa_x=\I_2$ and hence $c_x(E_{h(w)})=2$.   Reducing the Weierstrass equation for $E_{h(w)}$ over $\FF_x$, we have 
\[
-h(w) \yy^2= \xx^3-12f(0)^2 \xx+16f(0)^3 = 6f(0)(\xx-2f(0))^2+(\xx-2f(0))^3.
\]
So $E_{h(w)}$ has split multiplicative reduction at $x$ if and only if $-6h(w)f(0)$ is a square in $\FF_p$.  Since $h(u)=(2u/(u^2+1))^2$, we find that $-h(w)f(0)=h(w)\prod_{i=1}^{4n-1} h(i+1)$ is a non-zero square in $\FF_p$.  So $E_{h(w)}$ has split multiplicative reduction at $x$ if and only if $6$ is a square in $\FF_p$.
\item
Suppose $x=1$.  We have $\kappa_x=\III^*$ and hence $c_x(E_{h(w)})=2$.
\item
Suppose $x=a$ is a root of $(t-h(w))f(t) \bmod{p} \in \FF_p[t]$.  We have $\kappa_x=\I_0^*$.   Tate's algorithm shows that $c_x(E_{h(w)})=1+m$ where $m$ is the number of roots of 
\begin{align*}
P(\xx)&:= \xx^3- 3 (a-1)^3 (a-4) \xx -2(a-1)^5(a+8)
\end{align*} 
in $\FF_p$.  The polynomial $P(\xx)$ has root $2(a-1)^2$, and $P(\xx)/(\xx-2(a-1)^2)$ equals
\[
Q(\xx):=\xx^2+2(a-1)^2 \xx + (a-1)^3(a+8).
\]
Using that $a=h(b)$ for some $b\in \FF_p$, we find that the discriminant of $Q(\xx)$ is a square (moreover, it equals $6^2(b-1)^6(b+1)^6/(b^2+1)^6$), so $m=3$.   Therefore, $c_x(E_{h(w)})=4$.
\end{itemize}
Note that the curve $E_{h(w)}$ has multiplicative reduction at a closed point $x\neq\infty$ and Kodaira symbol $\I_0^*$ at more than one closed point $x\neq \infty$.   

From the computations above, we find that $N_{E_{h(w)}}=-4+1+2+1+(4n)\cdot 2 =8n$, $\chi_{E_{h(w)}} = (2+9+1+6(4n))/12= 2n+1$, $\calL_{E_{h(w)}}=1$, $\gamma_{E_{h(w)}}=1$ and $B_{E_{h(w)}}=1$.  Since this hold for all large primes $p$ and all $w\in W(\FF_p)$, we thus have $N=8n$, $\chi=2n+1$, $\calL=1$, $\gamma=1$ and $B=1$.    Observe that $N$ is even, $\chi$ is odd and $\gamma$ is a square modulo $\ell$.   We have $6B=6\leq 8n=N$.   

The above computations show that $c_{E_{h(w)}}$ is a power of $4$, and hence a square modulo $\ell$, and that
\begin{align*}
\varepsilon_{E_{h(w)}} &=  \legendre{6}{p} \legendre{-3}{p} \legendre{-2}{p} \legendre{-1}{p}^{4n}= 1.
\end{align*}

We have verified all the conditions of \S\ref{S:criterion} and in particular that assumption (\ref{C}) holds.  From Theorem~\ref{T:main criterion}, we deduce that $\Omega(V)$ occurs as the Galois group of a regular extension of $\QQ(t)$ where $V$ is an orthogonal space over $\FF_\ell$ of dimension $N=8n$ with $n\geq1$ and $\det(V)=(\FF_\ell^\times)^2$.

\subsection{\texorpdfstring{$N\equiv 2 \pmod{8}$}{N congruent to 2 mod 8}}
Take any integer $n\geq 1$.   Define the rational function $h(u)=  -(u^2-1)^2/(4u^2)$ and the polynomial  $f(t)=\prod_{i=1}^{4n}(t-h(i+1))$.  Consider the Weierstrass equation
\[
\yy^2=\xx^3-3(t-1)(t-4)f(t)^2 \xx-2(t-1)^2(t+8)f(t)^3;
\]
it has discriminant $\Delta(t)= -2^{6} 3^6 f(t)^6 t^2(t-1)^3$ and the $j$-invariant of the corresponding elliptic curve over $\QQ(t)$ is $-2^{6} (t-4)^3 t^{-2}$.  

Now take notation as in \S\ref{S:criterion}.  Take any prime $p\nmid 6\ell$ such that $f(t)$ modulo $p$ is well-defined and separable, and $f(0)f(1)\not\equiv 0 \pmod{p}$.   Take any $w\in W(\FF_p)$, i.e., any $w\in \FF_p$ such that $w\neq 0$ and  $\Delta(h(w))\neq 0$.  Let $x$ be any closed point of $\PP^1_{\FF_p}=\Spec \FF_p[t] \cup \{\infty\}$ for which $E_{h(w)}/\FF_p(t)$ has bad reduction and let $\kappa_x$ be the Kodaira symbol of $E_{h(w)}/\FF_p(t)$ at $x$.

\begin{itemize}
\item
Suppose $x=\infty$. We have $\kappa_x=\I_1$, so $c_x(E_{h(w)})=1$.   Over $\FF_p(t)_x=\FF_p(\!(t^{-1})\!)$, the elliptic curve $E_{h(w)}$ is isomorphic to the curve defined by the Weierstrass equation $\yy^2=\xx^3-3(1-t^{-1})(1-4t^{-1}) \xx-2(1-t^{-1})^2(1+8t^{-1})$.  Reducing, we have the equation
\[
\yy^2= \xx^3-3\xx-2 = -3(\xx+1)^2+(\xx+1)^3
\]
over $\FF_x$. The curve $E_{h(w)}$ has split multiplicative reduction at $x$ if and only if $-3$ is a square in $\FF_p$.
\item
Suppose $x=0$. We have $\kappa_x=\I_2$, so $c_x(E_{h(w)})=2$.   Reducing the Weierstrass equation for $E_{h(w)}$ over $\FF_x$, we have 
\[
-h(w) \yy^2= \xx^3-12f(0)^2 \xx-16f(0)^3 = -6f(0)(\xx+2f(0))^2+(\xx+2f(0))^3.
\]
So $E_{h(w)}$ has split multiplicative reduction at $x$ if and only if $6h(w)f(0)$ is a square in $\FF_p$.  Since $-h(u)=((u^2-1)/(2u))^2$ and $-h(w)f(0)=-h(w) \prod_{i=1}^{4n}(-h(i+1))$, we deduce that $E_{h(w)}$ has split multiplicative reduction at $x$ if and only if $-6$ is a square in $\FF_p$.
\item
Suppose $x=1$.  We have $\kappa_x=\III$, so $c_x(E_{h(w)})=2$.  
\item
Suppose $x=a$ is a root of $(t-h(w))f(t) \bmod{p} \in \FF_p[t]$.  We have $\kappa_x=\I_0^*$.    Tate's algorithm shows that $c_x(E_{h(w)})=1+m$ where $m$ is the number of roots of 
\[
P(\xx):=\xx^3-3(a-1)(a-4) \xx-2(a-1)^2(a+8)
\]
in $\FF_p$.   The polynomial $P(\xx)$ has root $2(a-1)$, and $P(\xx)/(\xx-2(a-1))$ equals
\[
Q(\xx):=\xx^2 + 2(a - 1)\xx + (a-1)(a+8).
\]
Using that $a=h(b)$ for some $b\in \FF_p$, we find that the discriminant of $Q(\xx)$ is a square (moreover, it equals $3^2(b^2+1)^2/b^2$), so $m=3$.   Therefore, $c_x(E_{h(w)})=4$.
\end{itemize}
Note that the curve $E_{h(w)}$ has multiplicative reduction at a closed point $x\neq\infty$ and Kodaira symbol $\I_0^*$ at more than one closed point $x\neq \infty$.   

From the computations above, we find that $N_{E_{h(w)}}=-4+1+2+1+(4n+1)\cdot 2 =8n+2$, $\chi_{E_{h(w)}} = (2+3+1+6(4n+1))/12= 2n+1$, $\calL_{E_{h(w)}}=1$, $\gamma_{E_{h(w)}}=1$ and $B_{E_{h(w)}}=1$.  Since this hold for all large primes $p$ and all $w\in W(\FF_p)$, we thus have $N=8n+2$, $\chi=2n+1$, $\calL=1$, $\gamma=1$ and $B=1$.    Observe that $N$ is even, $\chi$ is odd and $\gamma$ is a square modulo $\ell$.   We have $6B=6\leq 8n+2=N$.   

The above computations show that $c_{E_{h(w)}}$ is a power of $4$, and hence a square modulo $\ell$, and that 
\begin{align*}
\varepsilon_{E_{h(w)}} &=  \legendre{-3}{p} \legendre{-6}{p} \legendre{-2}{p} \legendre{-1}{p}^{4n+1}= 1.
\end{align*}

We have verified all the conditions of \S\ref{S:criterion} and in particular that assumption (\ref{C}) holds.  From Theorem~\ref{T:main criterion}, we deduce that $\Omega(V)$ occurs as the Galois group of a regular extension of $\QQ(t)$ where $V$ is an orthogonal space over $\FF_\ell$ of dimension $N=8n+2$ with $n\geq 1$ and $\det(V)=(\FF_\ell^\times)^2$.

\subsection{\texorpdfstring{$N\equiv 4 \pmod{8}$}{N congruent to 4 mod 8}}
Take any integer $n\geq 1$.  Define the rational function $h(u)= -3u^2$ and the polynomial  $f(t)=\prod_{i=1}^{4n}(t-h(i))$.   Consider the Weierstrass equation
\[
\yy^2=\xx^3-3(t-1)(t-9)f(t)^2 \xx-2(t-1)(t-3)(t-9)f(t)^3;
\]
it has discriminant $\Delta(t)= -2^{8} 3^3 f(t)^6 t(t-1)^2(t-9)^2$ and the $j$-invariant of the corresponding elliptic curve over $\QQ(t)$ is $-2^4 3^3 (t-1)(t-9) t^{-1}$.  

Now take notation as in \S\ref{S:criterion}.  Take any prime $p\nmid 6\ell$ such that $f(t)$ modulo $p$ is well-defined and separable, and $f(0)f(1)f(9)\not\equiv 0 \pmod{p}$.   Take any $w\in W(\FF_p)$, i.e., any $w\in \FF_p$ for which $\Delta(h(w))\neq 0$.  Let $x$ be any closed point of $\PP^1_{\FF_p}=\Spec \FF_p[t] \cup \{\infty\}$ for which $E_{h(w)}/\FF_p(t)$ has bad reduction and let $\kappa_x$ be the Kodaira symbol of $E_{h(w)}/\FF_p(t)$ at $x$.

\begin{itemize}
\item
Suppose $x=\infty$.  We have $\kappa_x=\I_1$, so $c_x(E_{h(w)})=1$.  Over $\FF_p(t)_x=\FF_p(\!(t^{-1})\!)$, the elliptic curve $E_{h(w)}$ is isomorphic to the curve defined by the Weierstrass equation
\[
\yy^2=\xx^3-3(1-t^{-1})(1-9t^{-1}) \xx-2(1-t^{-1})(1-3t^{-1})(1-9t^{-1}).
\]
Reducing, we have the equation $\yy^2= \xx^3-3\xx-2 = -3(\xx+1)^2+(\xx+1)^3$ over $\FF_x$.    The curve $E_{h(w)}$ has split multiplicative reduction at $x$ if and only if $-3$ is a square in $\FF_p$.
\item
Suppose $x=0$. We have $\kappa_x=\I_1$, so $c_x(E_{h(w)})=1$.   Reducing the Weierstrass equation for $E_{h(w)}$ over $\FF_x$, we have 
\[
-h(w) \yy^2= \xx^3-27f(0)^2 \xx+54f(0)^3 = 9f(0)(\xx-3f(0))^2+(\xx-3f(0))^3.
\]
So $E_{h(w)}$ has split multiplicative reduction at $x$ if and only if $-h(w)f(0)$ is a square in $\FF_p$.  Since $-h(w)f(0)=3 w^2 \prod_{i=1}^{4n}(-3 i^2)=3 (3^{2n} w \prod_{i=1}^{4n} i)^2$, we deduce that $E_{h(w)}$ has split multiplicative reduction at $x$ if and only if $3$ is a square in $\FF_p$.
\item
Suppose $x=1$.  We have $\kappa_x=\II$, so $c_x(E_{h(w)})=1$.
\item
Suppose $x=9$.  We have $\kappa_x= \II$, so $c_x(E_{h(w)})=1$.
\item
Suppose $x=a$ is a root of $(t-h(w))f(t) \bmod{p} \in \FF_p[t]$.  We have $\kappa_x=\I_0^*$.  Tate's algorithm shows that $c_x(E_{h(w)})=1+m$ where $m$ is the number of roots of 
\[
P(\xx):=\xx^3-3(a-1)(a-9)\xx-2(a-1)(a-3)(a-9)
\] 
in $\FF_p$.  Using that $a=h(b)$ for some $b\in \FF_p$, we find that the discriminant of $P(\xx)$ is a square (moreover, it equals $2^{2} 3^4 b(b^2+3)(3b^2 + 1)$ squared), so $m$ equals $0$ or $3$.   Therefore, $c_x(E_{h(w)})$ equals $1$ or $4$.
\end{itemize}
Note that the curve $E_{h(w)}$ has multiplicative reduction at a closed point $x\neq\infty$ and Kodaira symbol $\I_0^*$ at more than one closed point $x\neq \infty$.   

From the computations above, we find that $N_{E_{h(w)}}=-4+1+1+2+2+(4n+1)\cdot 2 =8n+4$, $\chi_{E_{h(w)}} = (1+1+2+2+6(4n+1))/12= 2n+1$, $\calL_{E_{h(w)}}=1$, $\gamma_{E_{h(w)}}=1$ and $B_{E_{h(w)}}=2$.  Since this hold for all large primes $p$ and all $w\in W(\FF_p)$, we thus have $N=8n+4$, $\chi=2n+1$, $\calL=1$, $\gamma=1$ and $B=2$.    Observe that $N$ is even, $\chi$ is odd and $\gamma$ is a square modulo $\ell$.   We have $6B=12\leq 8n+4=N$.   

The above computations show that $c_{E_{h(w)}}$ is a power of $4$, and hence a square modulo $\ell$, and that 
\begin{align*}
\varepsilon_{E_{h(w)}} &=  \legendre{3}{p} \legendre{-3}{p} \legendre{-1}{p}^2 \legendre{-1}{p}^{4n+1}= 1.
\end{align*}

We have verified all the conditions of \S\ref{S:criterion} and in particular that assumption (\ref{C}) holds.  From Theorem~\ref{T:main criterion}, we deduce that $\Omega(V)$ occurs as the Galois group of a regular extension of $\QQ(t)$ where $V$ is an orthogonal space over $\FF_\ell$ of dimension $N=8n+4$ with $n\geq 1$ and $\det(V)=(\FF_\ell^\times)^2$.

\subsection{\texorpdfstring{$N\equiv 6 \pmod{8}$}{N congruent to 6 mod 8}}
Take any integer $n\geq 0$.  Define the rational function $h(u)= (u^2+1)/(2u)$ and the polynomial  $f(t)=\prod_{i=1}^{4n+1}(t-h(i+1))$.  Consider the Weierstrass equation
\begin{align} \label{E:6}
\yy^2 &= (\xx- t(t^2-2)f(t))\cdot (\xx-t(t^2+1)f(t))\cdot (\xx+t(2t^2-1)f(t)) \\
&= \xx^3 -3t^2(t^4 -t^2 +1)f(t)^2\xx + t^3(2t^6 - 3t^4 - 3t^2 + 2)f(t)^3;  \notag
\end{align}
it has discriminant $\Delta(t)=2^4 3^6  t^{10} (t-1)^2(t+1)^2$ and the $j$-invariant of the corresponding elliptic curve over $\QQ(t)$ is $2^8(t^4-t^2+1)^3 t^{-4} (t-1)^{-2}(t+1)^{-2}$.  

Now take notation as in \S\ref{S:criterion}.  Take any prime $p\nmid 6\ell$ such that $f(t)$ modulo $p$ is well-defined and separable, and $f(0)f(1)f(-1)\not\equiv 0 \pmod{p}$.   Take any $w\in W(\FF_p)$, i.e., any $w\in W(\FF_p)$ for which $w\neq 0$ and $\Delta(h(w))\neq 0$.  Let $x$ be any closed point of $\PP^1_{\FF_p} = \Spec \FF_p[t]\cup \{\infty\}$ for which $E_{h(w)}/\FF_p(t)$ has bad reduction at $x$.

\begin{itemize}
\item
Suppose $x=1$.  We have $\kappa_x=\I_2$, so $c_x(E_{h(w)})=2$.   Reducing the equation to $\FF_x$, we have 
\begin{align*}
(1-h(w)) \yy^2 &=(\xx+f(1))^2(\xx-2f(1))\\ &=-3f(1)(\xx+f(1))^2+(\xx+f(1))^3.
\end{align*}
So $E_{h(w)}$ has split multiplicative reduction at $x$ if and only if $-3(1-h(w))f(1)$ is a square in $\FF_p^\times$.
\item
Suppose $x=-1$.   We have $\kappa_x=\I_2$, so $c_x(E_{h(w)})=2$.   Reducing the equation to $\FF_x$, we have 
\begin{align*}
(-1-h(w)) \yy^2 &=(\xx-f(-1))^2(\xx+2f(-1))\\ &=3f(-1)(\xx-f(-1))^2+(\xx-f(-1))^3.
\end{align*}
So $E_{h(w)}$ has split multiplicative reduction at $x$ if and only if $-3(h(w)+1)f(-1)$ is a square in $\FF_p^\times$.
\item
Suppose $x=0$.   We have $\kappa_x=\I_4^*$.  Using that $p$ is odd and $f(0)\neq 0$, we find that $(t-h(w))f(t)$ belongs to $-h(w)f(0)\cdot (\FF_p(\!(t)\!)^\times)^2$.  Using Tate's algorithm, one can then show that $c_x(E_{h(w)})$ is a power of $2$ and equals the Tamagawa number of the elliptic curve 
\[
\yy^2= \xx^3 -3(t^2 -1 +t^{-2}) \xx + (2t^3 - 3t - 3t^{-1} + 2t^{-3})
\]
over $\FF_p(\!(t)\!)$.
\item
Suppose $x=\infty$.  We have $\kappa_x=\I_4^*$.  Using that $p$ is odd and $f(t)$ is monic of odd degree, we find that $(t-h(w))f(t)$ belongs to $(\FF_p(\!(t^{-1})\!)^\times)^2$.  So $c_x(E_{h(w)})$ equals the Tamagawa number of the elliptic curve 
\[
\yy^2= \xx^3 -3(t^2 -1 +t^{-2}) \xx + (2t^3 - 3t - 3t^{-1} + 2t^{-3})
\]
over $\FF_p(\!(t^{-1})\!)$.
\item
Suppose $x=a$ is a root of $(t-h(w))f(t) \bmod{p} \in \FF_p[t]$.  We have $\kappa_x=\I_0^*$.   We have $c_x(E_{h(w)})=4$ since the polynomial of $\xx$ in (\ref{E:6}) splits into linear factors.
\end{itemize}
Note that the curve $E_{h(w)}$ has multiplicative reduction at a closed point $x\neq\infty$ and Kodaira symbol $\I_0^*$ at more than one closed point $x\neq \infty$.   

From the computations above, we find that $N_{E_{h(w)}}=-4+2+2+1+1+(4n+2)\cdot 2 =8n+6$, $\chi_{E_{h(w)}} = (10+10+2+2+6(4n+2))/12= 2n+3$, $\calL_{E_{h(w)}}=1$, $\gamma_{E_{h(w)}}=1$ and $B_{E_{h(w)}}=1$.  Since this hold for all large primes $p$ and all $w\in W(\FF_p)$, we thus have $N=8n+4$, $\chi=2n+1$, $\calL=1$, $\gamma=1$ and $B=1$.    Observe that $N$ is even, $\chi$ is odd and $\gamma$ is a square modulo $\ell$.   We have $6B=6\leq 8n+6=N$.   

Note that the Tamagawa numbers described in the cases $x=0$ and $x=\infty$ are equal, since the given curves are isomorphic with respect to the isomorphism $\FF_p(\!(t)\!)\xrightarrow{\sim} \FF_p(\!(t^{-1})\!)$, $\alpha(t)\mapsto \alpha(t^{-1})$ of base fields.  Therefore, the product $\prod_x c_x(E_{h(w)})$ is a power of $4$ and is thus a square modulo $\ell$.  We also have
\begin{align*}
\varepsilon_{E_{h(w)}} &=  \legendre{-3(1-h(w))f(1)}{p} \legendre{3(-1-h(w))f(-1)}{p} \legendre{-1}{p}^2 \legendre{-1}{p}^{4n+1}\\
&= \legendre{(1-h(w))(-1-h(w))}{p} \prod_{i=1}^{4n+1}\legendre{(1-h(i+1))(-1-h(i+1))}{p} =1
\end{align*}
where the last equality uses that $(1-h(u))(-1-h(u))= (u-1)^2(u+1)^2/(2u)^2$.

We have verified all the conditions of \S\ref{S:criterion} and in particular that assumption (\ref{C}) holds.  From Theorem~\ref{T:main criterion}, we deduce that $\Omega(V)$ occurs as the Galois group of a regular extension of $\QQ(t)$ where $V$ is an orthogonal space over $\FF_\ell$ of dimension $N=8n+6$ with $n\geq 0$ and $\det(V)=(\FF_\ell^\times)^2$.

\section{\texorpdfstring{Proof of Theorem~\ref{T:Omega total}($iv$)}{Proof of Theorem~\ref{T:Omega total}(iv)}} \label{SS:numerics even 2}

Take any prime $\ell\geq 5$ such that $2$, $3$, $5$ or $7$ is a non-square modulo $\ell$.

For an even integer $N\geq 6$, let $V$ be an orthogonal space of dimension $N$ over $\FF_\ell$ satisfying $\det(V)\neq (\FF_\ell^\times)^2$.  We use the criterion of \S\ref{S:criterion} to show that $\Omega(V)$, and hence $\POmega(V)$, occurs as the Galois group of a regular extension of $\QQ(t)$.   We proved the case with dimension $N$ and $\det(V)=(\FF_\ell^\times)^2$ in the pervious section.

\subsection{Case 1}  \emph{Suppose that $2$ is not a square modulo $\ell$.} 

Take any integer $n\geq 2$.   Define the rational function $h(u)= 1/(2u^2 + 1)$ and the polynomial  $f(t)=\prod_{i=1}^{n-1}(t-g(i))$ where $g(u)=1/(u^2+1)$.  Consider the Weierstrass equation
\[
\yy^2=\xx^3+3(t-1)(t-4)(3t-4)f(t)^2 \xx-4(t-1)^2(9t^2-32t+32)f(t)^3    
\]
it has discriminant $\Delta(t)= -2^{6} 3^6 f(t)^6 t^4(t-1)^3(t-2)^2$ and the $j$-invariant of the corresponding elliptic curve over $\QQ(t)$ is $1728(t-4)^3(t-4/3)^3(t-2)^{-2} t^{-4}$.  

Now take notation as in \S\ref{S:criterion}.  Take any prime $p\nmid 6\ell$ such that $f(t)$ modulo $p$ is well-defined and separable, and $f(0)f(1)f(2)\not\equiv 0 \pmod{p}$.   Take any $w\in W(\FF_p)$, i.e., any $w\in \FF_p$ for which $2w^2+1\neq 0$ and $\Delta(h(w))\neq 0$.  Let $x$ be any closed point of $\PP^1_{\FF_p}=\Spec \FF_p[t] \cup \{\infty\}$ for which $E_{h(w)}/\FF_p(t)$ has bad reduction and let $\kappa_x$ be the Kodaira symbol of $E_{h(w)}/\FF_p(t)$ at $x$.  

\begin{itemize}
\item
Suppose $x=\infty$. The symbol $\kappa_x$ is either $\III$ or $\III^*$.
\item
Suppose $x=1$.  We have $\kappa_x=\III$.
\item
Suppose $x=0$.  We have $\kappa_x=\I_4$.    Reducing the Weierstrass equation for $E_{h(w)}$ to an equation over $\FF_x$, we have 
\[
-h(w) \yy^2 = \xx^3-48 f(0)^2 \xx-128f(0)^3=-12f(0)(\xx+4f(0))^2+(\xx+4f(0))^3.
\]
So $E_{h(w)}$ has split multiplicative reduction at $x$ if and only if $3h(w)f(0)$ is a square in $\FF_p$. 
\item
Suppose $x=2$.  We have $\kappa_x=\I_2$.   Reducing the Weierstrass equation for $E_{h(w)}$ to an equation over $\FF_x$, we have 
\[
(1-h(w)) \yy^2 = \xx^3-12 f(1)^2 \xx-16f(1)^3=-6f(1)(\xx+2f(1))^2+(\xx+2f(1))^3.
\]
So $E_{h(w)}$ has split multiplicative reduction at $x$ if and only if $-6(1-h(w))f(1)$ is a square in $\FF_p$. 
\item
Suppose $x=a$ is an irreducible factor of $(t-h(w))f(t) \bmod{p} \in \FF_p[t]$.   We have $\kappa_x=\I_0^*$.
\end{itemize}
Note that the curve $E_{h(w)}$ has multiplicative reduction at a closed point $x\neq\infty$ and Kodaira symbol $\I_0^*$ at more than one closed point $x\neq \infty$.   

From the computations above, we find that $N_{E_{h(w)}}=-4+2+2+1+1+n\cdot 2 =2n+2$, $\calL_{E_{h(w)}}=1$, $\gamma_{E_{h(w)}}=2$ and $B_{E_{h(w)}}=1$.  Since this hold for all large primes $p$ and all $w\in W(\FF_p)$, we thus have $N=2n+2$, $\calL=1$, $\gamma=2$ and $B=1$.    By our assumption on $\ell$, $\gamma=2$ is a non-square modulo $\ell$.   We have $6B=6\leq 2n+2=N$.   

From our computations, we have
\begin{align*}
\varepsilon(E_{h(w)}) &=  \legendre{-2}{p}^2 \legendre{3h(w)f(0)}{p} \legendre{-6(1-h(w))f(1)}{p} \legendre{-1}{p}^{n}\\
&= \legendre{2h(w)(1-h(w))}{p} \prod_{i=1}^{n-1}\legendre{g(i)(1-g(i))}{p}  = 1,
\end{align*}
where the last equality uses that $2h(u)(1-h(u))= (2u)^2/(2u^2+1)^2$ and $g(u)(1-g(u))=u^2/(u^2+1)^2$.  

We have verified all the conditions of \S\ref{S:criterion} and in particular that assumption (\ref{B}) holds.  From Theorem~\ref{T:main criterion}, we deduce that $\Omega(V)$ occurs as the Galois group of a regular extension of $\QQ(t)$ where $V$ is an orthogonal space over $\FF_\ell$ of dimension $N=2n+2$ with $n\geq 2$ and $\det(V)\neq (\FF_\ell^\times)^2$.

\subsection{Case 2}  \emph{Suppose that $3$ is not a square modulo $\ell$ and that $2$ is a square modulo $\ell$.}

Take any integer $n\geq 2$.  Define the rational function $h(u)= 1/(2u^2 + 1)$ and the polynomial  $f(t)=\prod_{i=1}^{n-1}(t-g(i))$ where $g(u)=1/(u^2+1)$.  Consider the Weierstrass equation
\begin{align*} 
\yy^2=\xx^3-&3(28t - 1)(147t^2 + 112t - 16)f(t)^2\xx \\-&2(28t - 1)(21609t^3 - 3430t^2 + 1568t - 64)f(t)^3; \notag
\end{align*}
it has discriminant $\Delta(t)= 2^{12} 3^6 7^9 f(t)^6 t^4(t-1)^3(t-1/28)^2$ and the $j$-invariant of the corresponding elliptic curve over $\QQ(t)$ is $2^6 7^{-6}(t -1/28)(147t^2 + 112t - 16)^3  t^{-4}(t-1)^{-3}$.    

Now take notation as in \S\ref{S:criterion}.  Take any prime $p\nmid 42\ell$ such that $f(t)$ modulo $p$ is well-defined and separable, and $f(0)f(1)f(1/28)\not\equiv 0 \pmod{p}$.   Take any $w\in W(\FF_p)$, i.e., any $w\in \FF_p$ for which $2w^2+1\neq 0$ and $\Delta(h(w))\neq 0$.  Let $x$ be any closed point of $\PP^1_{\FF_p}=\Spec \FF_p[t] \cup \{\infty\}$ for which $E_{h(w)}/\FF_p(t)$ has bad reduction and let $\kappa_x$ be the Kodaira symbol of $E_{h(w)}/\FF_p(t)$ at $x$.

\begin{itemize}
\item
Suppose $x=\infty$. The symbol $\kappa_x$ is $\III$ or $\III^*$.
\item
Suppose $x=1/28$.  We have $\kappa_x=\II$.
\item
Suppose $x=0$.  We have $\kappa_x=\I_4$. Reducing the Weierstrass equation for $E_{h(w)}$ to an equation over $\FF_x$, we have 
\[
-h(w) \yy^2 = \xx^3-48 f(0)^2 \xx-128f(0)^3=-12f(0)(\xx+4f(0))^2+(\xx+4f(0))^3.
\]
So $E_{h(w)}$ has split multiplicative reduction at $x$ if and only if $3h(w)f(0)$ is a square in $\FF_p$. 
\item
Suppose $x=1$.  We have $\kappa_x=\I_3$.   Reducing the Weierstrass equation for $E_{h(w)}$ to an equation over $\FF_x$, we have 
\begin{align*}
(1-h(w)) \yy^2 &= \xx^3-19683f(1)^2 \xx-1062882 f(1)^3\\
&=-243f(1)(\xx+81f(1))^2+(\xx+81f(1))^3.
\end{align*}
So $E_{h(w)}$ has split multiplicative reduction at $x$ if and only if $-3(1-h(w))f(1)$ is a square in $\FF_p$. 
\item
Suppose $x=a$ is a root of $(t-h(w))f(t) \bmod{p} \in \FF_p[t]$.  We have $\kappa_x=\I_0^*$.
\end{itemize}
Note that the curve $E_{h(w)}$ has multiplicative reduction at a closed point $x\neq\infty$ and Kodaira symbol $\I_0^*$ at more than one closed point $x\neq \infty$.   

From the computations above, we find that $N_{E_{h(w)}}=-4+2+2+1+1+n\cdot 2  =2n+2$, $\calL_{E_{h(w)}}=1$, $\gamma_{E_{h(w)}}=6$ and $B_{E_{h(w)}}=1$.  Since this hold for all large primes $p$ and all $w\in W(\FF_p)$, we thus have $N=2n+2$, $\calL=1$, $\gamma=6$ and $B=1$.    By our assumptions on $\ell$, $\gamma=6$ is a non-square modulo $\ell$.   We have $6B=6\leq 2n+2=N$.   

From our computations, we have
\begin{align*}
\varepsilon_{E_{h(w)}} &=  \legendre{-1}{p} \legendre{-2}{p} \legendre{3h(w)f(0)}{p} \legendre{-3(1-h(w))f(1)}{p} \legendre{-1}{p}^{n}\\
&= \legendre{2h(w)(1-h(w))}{p} \prod_{i=1}^{n-1}\legendre{g(i)(1-g(i))}{p}  = 1
\end{align*}
where the last equality uses that $2h(u)(1-h(u))= (2u)^2/(2u^2+1)^2$ and $g(u)(1-g(u))=u^2/(u^2+1)^2$.  

We have verified all the conditions of \S\ref{S:criterion} and in particular that assumption (\ref{B}) holds.  From Theorem~\ref{T:main criterion}, we deduce that $\Omega(V)$ occurs as the Galois group of a regular extension of $\QQ(t)$ where $V$ is an orthogonal space over $\FF_\ell$ of dimension $N=2n+2$ with $n\geq 2$ and $\det(V)\neq (\FF_\ell^\times)^2$.

\subsection{Case 3}   \emph{Suppose that $5$ is not a square modulo $\ell$ and that $2$ and $3$ are squares modulo $\ell$}.

Take any integer $n\geq 2$.  Let $h(u)$ be the rational function $15/(-u^2+15)$ if $n$ is even and $5/(-u^2+5)$ if $n$ is odd.   Define the polynomial  $f(t)=\prod_{i=1}^{n-1}(t-g(i))$ where $g(u)=1/(u^2+1)$.  Consider the Weierstrass equation
\begin{align*} 
\yy^2=\xx^3+&3(135t^2 + 96t - 256)f(t)^2 \xx\\ -&2(486t^4 + 621t^3 - 3024t^2 - 2304t + 4096)f(t)^3; \notag
\end{align*}
it has discriminant $\Delta(t)= -2^{8} 3^{13} f(t)^6 t^5(t-1)(t+16/9)^2$ and the $j$-invariant of the corresponding elliptic curve over $\QQ(t)$ is $2^{4} 3^2 5^3(t+16/9)(t-16/15)^3 t^{-5}(t-1)^{-1}$.  

Now take notation as in \S\ref{S:criterion}.  Take any prime $p\nmid 30\ell$ such that $f(t)$ modulo $p$ is well-defined and separable, and $f(0)f(1)f(16/9)\not\equiv 0 \pmod{p}$.   Take any $w\in W(\FF_p)$, i.e., any $w\in \FF_p$ for which $-w^2+15\neq 0$ if $n$ is even, $-w^2+5\neq 0$ if $n$ is odd, and $\Delta(h(w))\neq 0$.   Let $x$ be any closed point of $\PP^1_{\FF_p}=\Spec \FF_p[t] \cup \{\infty\}$ for which $E_{h(w)}/\FF_p(t)$ has bad reduction and let $\kappa_x$ be the Kodaira symbol of $E_{h(w)}/\FF_p(t)$ at $x$.

\begin{itemize}
\item
Suppose $x=\infty$.   The symbol $\kappa_x$ is $\IV$ or $\II^*$ when $n$ is even or odd, respectively.  
\item
Suppose $x=0$.  We have $\kappa_x=\I_5$.   Reducing the Weierstrass equation for $E_{h(w)}$ to an equation over $\FF_x$, we have 
\[
-h(w) \yy^2 = \xx^3-768 f(0)^2 \xx-8192f(0)^3= -48f(0)(\xx+16f(0))^2+(\xx+16f(0))^3.
\]
So $E_{h(w)}$ has split multiplicative reduction at $x$ if and only if $3h(w)f(0)$ is a square in $\FF_p$. 
\item
Suppose $x=1$.  We have $\kappa_x=\I_1$.  Reducing the Weierstrass equation for $E_{h(w)}$ to an equation over $\FF_x$, we have 
\begin{align*}
(1-h(w)) \yy^2 &= \xx^3-75f(1)^2 \xx+250 f(1)^3=15f(1)(\xx-5f(1))^2+(\xx-5f(1))^3.
\end{align*}
So $E_{h(w)}$ has split multiplicative reduction at $x$ if and only if $15(1-h(w))f(1)$ is a square in $\FF_p$. 
\item
Suppose $x=-16/9$.  We  have $\kappa_x=\II$.
\item
Suppose $x=a$ is a root of $(t-h(w))f(t) \bmod{p} \in \FF_p[t]$.   We have $\kappa_x=\I_0^*$.
\end{itemize}
Note that the curve $E_{h(w)}$ has multiplicative reduction at a closed point $x\neq\infty$ and Kodaira symbol $\I_0^*$ at more than one closed point $x\neq \infty$.   

From the computations above, we find that $N_{E_{h(w)}}=-4+2+2+1+1+n\cdot 2  =2n+2$, $\calL_{E_{h(w)}}=5$ and $B_{E_{h(w)}}=1$.  Also $\gamma_{E_{h(w)}}$ is equal to $15$ or $5$ if $n$ is even or odd, respectively.   Since this hold for all large primes $p$ and all $w\in W(\FF_p)$, we thus have $N=2n+2$, $\calL=5$, $\gamma \in \{5,15\}$ and $B=1$.    By our assumptions on $\ell$, $\ell\nmid \calL$ and $\gamma$ is a non-square modulo $\ell$.  We have $6B=6\leq 2n+2=N$.   

If $n$ is even, then 
\begin{align*}
\varepsilon_{E_{h(w)}} &=  \legendre{-3}{p} \legendre{3h(w)f(0)}{p} \legendre{15(1-h(w))f(1)}{p} \legendre{-1}{p} \legendre{-1}{p}^{n}\\
&= \legendre{-15h(w)(1-h(w))}{p} \prod_{i=1}^{n-1}\legendre{g(i)(1-g(i))}{p}  = 1,
\end{align*}
where the last equality uses that $-15h(u)(1-h(u))= 15^2 u^2/(u^2-15)^2$ and $g(u)(1-g(u))=u^2/(u^2+1)^2$.

If $n$ is odd, then 
\begin{align*}
\varepsilon_{E_{h(w)}} &=  \legendre{-1}{p} \legendre{3h(w)f(0)}{p} \legendre{15(1-h(w))f(1)}{p} \legendre{-1}{p} \legendre{-1}{p}^{n}\\
&= \legendre{-5h(w)(1-h(w))}{p} \prod_{i=1}^{n-1}\legendre{g(i)(1-g(i))}{p}  = 1,
\end{align*}
where the last equality uses that $-5h(u)(1-h(u))= 5^2 u^2/(u^2-5)^2$ and $g(u)(1-g(u))=u^2/(u^2+1)^2$.

We have verified all the conditions of \S\ref{S:criterion} and in particular that assumption (\ref{B}) holds.  From Theorem~\ref{T:main criterion}, we deduce that $\Omega(V)$ occurs as the Galois group of a regular extension of $\QQ(t)$ where $V$ is an orthogonal space over $\FF_\ell$ of dimension $N=2n+2$ with $n\geq 2$ and $\det(V)\neq (\FF_\ell^\times)^2$.
\
\subsection{Case 4}   \emph{Suppose that $7$ is not a square modulo $\ell$ and that $2$, $3$ and $5$ are squares modulo $\ell$.}

Take any integer $n\geq 2$.  Let $h(u)$ be the rational function $63/4 \cdot (u^2 - 14)^{-1}$ if $n$ is even and $189/4\cdot (u^2 - 42)^{-1}$ if $n$ is odd.    Define the polynomial  $f(t)=\prod_{i=1}^{n-1}(t-g(i))$ where $g(u)=-9/8\cdot  (u^2+1)^{-1}$.  Consider the Weierstrass equation
\begin{align*} 
\yy^2=\xx^3-&12(9t+4)^2(14t^3+42t^2+36t+9)f(t)^2 \xx\\ -&24(9t+4)^3(8t^5+87t^4+222t^3+234t^2+108t+18)f(t)^3; \notag
\end{align*}
it has discriminant $\Delta(t)= -2^{10} 3^3 t^7 (8t+9)^2 (9t+4)^7$ and the $j$-invariant of the corresponding elliptic curve over $\QQ(t)$ is $-2^8 3^3 (14t^3+42t^2+36t+9)^3 t^{-7}(8t+9)^{-2}(9t+4)^{-1}$.   

Now take notation as in \S\ref{S:criterion}.  Take any prime $p\nmid 42\ell$ such that $f(t)$ modulo $p$ is well-defined and separable, and $f(0)f(-4/9)f(-9/8)\not\equiv 0 \pmod{p}$.   Take any $w\in W(\FF_p)$, i.e., any $w\in \FF_p$ for which $w^2-14\neq 0$ if $n$ is even, $w^2-42\neq 0$ if $n$ is odd, and $\Delta(h(w))\neq 0$.   Let $x$ be any closed point of $\PP^1_{\FF_p}=\Spec \FF_p[t] \cup \{\infty\}$ for which $E_{h(w)}/\FF_p(t)$ has bad reduction and let $\kappa_x$ be the Kodaira symbol of $E_{h(w)}/\FF_p(t)$ at $x$.

\begin{itemize}
\item
Suppose $x=\infty$.   The symbol $\kappa_x$ is  $\II$ or $\IV^*$ when $n$ is even or odd, respectively.  
\item
Suppose $x=-4/9$.  We have $\kappa=\I_1^*$.
\item
Suppose $x=0$.  We have $\kappa_x=\I_7$.   Reducing the Weierstrass equation for $E_{h(w)}$ to an equation over $\FF_x$, we have 
\[
-h(w) \yy^2 = \xx^3-1728 f(0)^2 \xx-27648f(0)^3= -72f(0)(\xx+24f(0))^2+(\xx+24f(0))^3.
\]
So $E_{h(w)}$ has split multiplicative reduction at $x$ if and only if $2h(w)f(0)$ is a square in $\FF_p$. 
\item
Suppose $x=-9/8$.  We have $\kappa=\I_2$.  Reducing the Weierstrass equation for $E_{h(w)}$ to an equation over $\FF_x$, we have 
\begin{align*}
(-9/8-h(w)) \yy^2 &= \xx^3-3176523/4096 \cdot f(-9/8)^2 \xx+1089547389/131072\cdot f(-9/8)^3\\
&=3087/64 f(-9/8) (\xx-1029/64\cdot f(-9/8))^2+(\xx-1029/64\cdot f(-9/8))^3.
\end{align*}
So $E_{h(w)}$ has split multiplicative reduction at $x$ if and only if $7(-9/8-h(w))f(-9/8)$ is a square in $\FF_p$. 
\item
Suppose $x=a$ is a root of $(t-h(w))f(t) \bmod{p} \in \FF_p[t]$.  We have $\kappa_x=\I_0^*$.
\end{itemize}
Note that the curve $E_{h(w)}$ has multiplicative reduction at a closed point $x\neq\infty$ and Kodaira symbol $\I_0^*$ at more than one closed point $x\neq \infty$.   

From the computations above, we find that $N_{E_{h(w)}}=-4+2+2+1+1+n\cdot 2  =2n+2$, $\calL_{E_{h(w)}}=7$ and $B_{E_{h(w)}}=1$.    Also $\gamma_{E_{h(w)}}$ is equal to $2\cdot 7$ or $6\cdot 7$ if $n$ is even or odd, respectively.     Since this hold for all large primes $p$ and all $w\in W(\FF_p)$, we thus have $N=2n+2$, $\calL=7$, $\gamma \in \{2\cdot 7,6\cdot 7\}$ and $B=1$.    By our assumptions on $\ell$, $\ell\nmid \calL$ and $\gamma$ is a non-square modulo $\ell$.  We have $6B=6\leq 2n+2=N$.   

If $n$ is even, then
\begin{align*}
\varepsilon_{E_{h(w)}} &=  \legendre{-1}{p}^2 \legendre{2h(w)f(0)}{p} \legendre{7(-9/8-h(w))f(-9/8)}{p} \legendre{-1}{p}^{n}\\
&= \legendre{14h(w)(9/8+h(w))}{p} \prod_{i=1}^{n-1}\legendre{g(i)(-9/8-g(i))}{p}  = 1,
\end{align*}
where the last equality uses that $14h(u)(9/8+h(u))= 3^4 7^2 4^{-2} u^2 (u^2-14)^{-2}$ and $g(u)(-9/8-g(u))=9^2 8^{-2} u^2(u^2+1)^{-2}$.

If $n$ is odd, then
\begin{align*}
\varepsilon_{E_{h(w)}} &= \legendre{-3}{p} \legendre{-1}{p} \legendre{2h(w)f(0)}{p} \legendre{7(-9/8-h(w))f(-9/8)}{p}  \legendre{-1}{p}^{n}\\
&= \legendre{42h(w)(9/8+h(w))}{p} \prod_{i=1}^{n-1}\legendre{g(i)(-9/8-g(i))}{p}  = 1,
\end{align*}
where the last equality uses that $42h(u)(9/8+h(u))= 3^6 7^2 4^{-2} u^2 (u^2-42)^{-2}$ and $g(u)(-9/8-g(u))=9^2 8^{-2} u^2(u^2+1)^{-2}$.

We have verified all the conditions of \S\ref{S:criterion} and in particular that assumption (\ref{B}) holds.  From Theorem~\ref{T:main criterion}, we deduce that $\Omega(V)$ occurs as the Galois group of a regular extension of $\QQ(t)$ where $V$ is an orthogonal space over $\FF_\ell$ of dimension $N=2n+2$ with $n\geq 2$ and $\det(V)\neq (\FF_\ell^\times)^2$.

\section{Proof of Theorem~\ref{T:Omega total} for \texorpdfstring{$n=5$}{n=5}}  \label{S:N=5}

Fix a prime $\ell\geq 5$.   Define the set $S=\{2,3,\ell\}$ and the ring $R=\ZZ[S^{-1}]$.  We use the construction of \S\ref{S:twists} with the Weierstrass equation
\[
\yy^2 = \xx^3+ 3(t^2-1)^3 \xx -2(t^2-1)^5
\]
being used for (\ref{E:initial WE}); its discriminant is $\Delta(t)=-2^6 3^3  t^2 (t-1)^9(t+1)^9$.  Define the $R$-scheme $M=\Spec R[u, \Delta(u)^{-1}]$.

Take any $m\in M(k)$, where $k$ is a finite field whose characteristic is not in $S$.   Let $E_m$ be the elliptic curve over $k(t)$ defined by the Weierstrass equation
\[
(t-m) \yy^2= \xx^3+ 3(t^2-1)^3 \xx -2(t^2-1)^5.
\]
The Kodaira symbol of $E_m$ at $t=0$, $1$, $-1$ and $m$ is $\I_2$, $\III^*$, $\III^*$ and $\I_0^*$, respectively.  The Kodaira symbol of $E_m$ at $\infty$ is $\II^*$.   Therefore, Lemma~\ref{L:independence of p} holds (without increasing $S$) and we have $\Phi=\{\I_2, \III^*, \III^*, \I_0^*, \II^*\}$, $N=5$, $\chi=3$, $\calL=1$, $\gamma=1$ and $B=2$.    In particular, $\ell \nmid 6\calL$.

By Proposition~\ref{P:main construction} (an examination of the proof shows that one does not need to increase $S$), there is an orthogonal space $V_\ell$ of dimension $5$ over $\FF_\ell$ and a continuous representation
\[
\theta_\ell \colon \pi_1(M) \to \Or(V_\ell)
\]
such that $\det(I-\theta_\ell(\Frob_m)T)\equiv L(T/q,E_m) \pmod{\ell}$ for any $m\in M(k)$, where $k$ is any finite field of order $q$ whose characteristic is not in $S$.  \\

Take $h(u)=(-u^2 + 3)/(u^2 + 3)$. As noted in Remark~\ref{R:foreshadow computations}, all the conditions of the criterion of \S\ref{S:criterion} hold \emph{except} for the following two:
\begin{itemize}
\item
For all sufficiently large primes $p$ and all $w\in W(\FF_p)$, $E_{h(w)}$ has Kodaira symbol $\I_0^*$ at more that one closed point of $\Aff^1_{\FF_p}$.
\item 
$6B \leq N$.
\end{itemize}

In the proof of Theorem~\ref{T:main criterion}, these two conditions are only used to prove that $\theta_\ell(\pi_1(M_{\Qbar})) \supseteq \Omega(V_\ell)$.  So to prove that $\Omega_5(\ell)$ occurs as the Galois group of a regular extension of $\QQ(t)$, it suffices to prove the following lemma.

\begin{lemma} \label{L:big monodromy 5}
The group $\theta_\ell(\pi_1(M_\Qbar))$ contains $\Omega(V_\ell)$ and is not a subgroup of $\SO(V_\ell)$.
\end{lemma}

We will give a proof of Lemma~\ref{L:big monodromy 5} in \S\ref{SS:proof of big monodromy 5}.   To rule out some of the possible small images of $\theta_\ell$, we will use the following computational result.

\begin{lemma} \label{L:Lfunction orders}
There is an element $g\in \theta_\ell(\pi_1(M))$ such $g^e\neq I$ for all $e\in \{16,20,24,28,36\}$.
\end{lemma}
\begin{proof}
First suppose that $\ell\neq 5$.  Take $m:=2$ in $M(\FF_5)$; we have an elliptic curve $E_2$ defined over $\FF_5(t)$.   One can compute that
\[
L(T/5,E_2)=-T^5 + 2/5\cdot T^4 - 1/25\cdot T^3 + 1/25\cdot T^2 - 2/5\cdot T + 1.
\]
One approach is to use the power series definition to compute the terms up to degree $5$ (less terms are required if you use the functional equation); we have verified this $L$-function  with \texttt{Magma}'s function \texttt{LFunction} \cite{Magma}.   Since $5\notin S$, we have $\det(I-\theta_\ell(\Frob_m) T) \equiv L(T/5,E_2) \pmod{\ell}$.  Let $A\in \GL_{5}(\ZZ[5^{-1}])$ be the companion matrix for the monic polynomial $-L(T/5,E_2)$.  The matrix $A$ modulo $\ell$ in $\GL_2(\FF_\ell)$ has the same characteristic polynomial as $-\theta_\ell(\Frob_m)$.

Take any $e\in\{16,20,24,28,36\}$.  Therefore, $A^e$ modulo $\ell$ and $\theta_\ell(\Frob_p)^e$ have the same characteristic polynomial.  In particular, if $\theta_\ell(\Frob_m)^e = I$, then all the coefficients of the polynomial $\det(TI-A^e)-(T-1)^5 \in \ZZ[1/5]$ are divisible by $\ell$.    An easy computer computation shows that this only happens when $\ell=17$ and $e=36$ (we are using that $\ell> 5$).    This shows that the lemma holds with $g=\theta_\ell(\Frob_m)$ when $\ell\notin\{5,17\}$.

A similar computation starting with $m=3$ in $M(\FF_7)$ can be used to prove the lemma for the excluded primes $\ell\in \{5,17\}$.  Similarly, we take $A \in \GL_5(\ZZ[7^{-1}])$ to be the companion matrix of 
\[
L(T/7,E_3)=1 - 33/49\cdot T^3 - 33/49\cdot T^2 + T^5; 
\]
you do not need to change the sign since the polynomial is monic.  
\end{proof}

\subsection{Proof of Lemma~\ref{L:big monodromy 5}} \label{SS:proof of big monodromy 5}
Define the groups $G:=\theta_\ell(\pi_1(M))$ and $G^g:=\theta_\ell(\pi_1(M_\Qbar))$.      

Using the work of Hall, as outlined in \S\ref{SS:Hall sketch}, we find that the group $G^g$ contains acts irreducibly on $V_\ell$ (by Lemma~\ref{L:tame and irreducible}) and contains a reflection (by Lemma~\ref{L:G gen}(\ref{L:monodromy c a}));  see Remark~\ref{R:foreshadow}.\\

Let $\calR$ be the group generated by all the reflections in $G^g$; it is a normal subgroup of both $G^g$ and $G$.  The group $\calR$ is non-trivial since $G^g$ contains a reflection.

\begin{lemma} \label{L:3.2 5}
Let $W$ be an irreducible $\calR$-submodule of $V_\ell$ and let $H$ be the subgroup of $G$ that stabilizes $W$.  The subspaces $\{gW\}_{g\in G/H}$ are pairwise orthogonal and we have $V_\ell=\oplus_{g\in G/H} \, g W$.
\end{lemma}
\begin{proof}
The proof of Lemma~3.2 of \cite{MR2372151} carries over verbatim (in the lemma of loc.~cit.~the role of $G$ is played by $G^g$).  Note that the subspace $gW$ is a $\calR$-module for all $g\in G$ since $\calR$ is a normal subgroup of $G$.   
\end{proof}

\begin{lemma} \label{L:irreducible 5}
The group $\calR$ acts irreducibly on $V_\ell$.
\end{lemma}
\begin{proof}
Suppose that $V_\ell$ is not an irreducible $\calR$-module and hence there is a proper irreducible $\calR$-submodule $W$ of $V_\ell$.   By Lemma~\ref{L:3.2 5}, we have $V_\ell = \oplus_{g\in G/H}\, gW$, where $H$ is the subgroup of $G$ that stabilizes $W$.   The vector space $V_\ell$ has dimension $N=5$, so $5=[G:H] \dim_{\FF_\ell} W$.  Since $W$ is a proper subspace of $V_\ell$, we deduce that $W$, and hence all the $gW$, have dimension $1$ over $\FF_\ell$.    We thus an orthogonal sum $V_\ell=\oplus_{i=1}^5 W_i$ with the group $G$ permuting the subspaces $W_i$.

Take any $g\in G$.  Since $g$ permutes the spaces $W_1,\ldots, W_5$, there is an integer $e\in \{4, 5, 6\}$ such that $g^e$ stabilizes each $W_i$.  The automorphism $g^e \in \Or(V_\ell)$ acts on $W_i$ and preserves the induced pairing on it.   We have $\Or(W_i) = \{\pm 1\}$ since $W_i$ has dimension $1$, so $g^{2e}$ acts trivially on each $W_i$.  Therefore, for any $g\in G$ we have $g^{e}=I$ for some $e\in \{8,10,12\}$.    However, this contradicts Lemma~\ref{L:Lfunction orders}.   Therefore, $V_\ell$ is an irreducible $\calR$-module.
\end{proof}

Since $\calR$ is generated by reflections and $V_\ell$ is an irreducible $\calR$-module by Lemma~\ref{L:irreducible 5}, we may use the classification of irreducible reflection groups as described by Zalesski{\u\i} and Sere{\v{z}}kin in \cite{MR603578}.   

Assume that $\Omega(V_\ell)$ is not a subgroup of $\calR$.  The classification in \cite{MR603578} shows that $\calR$ can be obtained by the reduction modulo $\ell$ of a finite irreducible reflection group of degree $5$ in characteristic $0$ and that the action of $\calR$ on $V_\ell$ is absolutely irreducible.

From Lemma~3.7 of \cite{MR2372151} and the classification in \cite{MR603578}, we find that $\calR$ is conjugate in $\GL(V_\ell)\cong \GL_5(\FF_\ell)$ to one of the groups denoted in \cite{MR603578} by $G_\ell(2,2,5)$, $G_\ell(2,1,5)$, $W_\ell(A_5)$ or $W_\ell(K_5)$.   

\begin{itemize}
\item
The group $G_\ell(2,2,5)$ is the subgroup of $\GL_5(\FF_\ell)$ generated by the permutation matrices and the diagonal matrix whose diagonal is $(-1,-1,1,1,1)$; it is isomorphic to the Weyl group of $W(D_5)$. 
\item
The group $G_\ell(2,1,5)$ is the subgroup of $\GL_5(\FF_\ell)$ generated by $G_\ell(2,2,5)$ and the matrix $-I$; it is isomorphic to the Weyl group of $W(C_5)$.    
\item
The group $W_\ell(A_5)$ is isomorphic to the symmetric group $\mathfrak{S}_6$ when $\ell\neq 7$.
\item
The group $W_\ell(A_5)$ is isomorphic to the symmetric group $\mathfrak{S}_7$ when $\ell=7$.
\item
The group $W_\ell(K_5)$ is isomorphic to $\{\pm 1\} \times \Omega_5(3)$.
\end{itemize}

A group theory computation shows that for any of the possibilities for $\calR$ given above, the outer automorphism group $\operatorname{Out}(G)$ of $\calR$ has cardinality at most $2$.   One can also verify that for any $r\in \calR$, there is an integer $e\in\{ 8, 10, 12, 14, 18\}$ such $r^e=I$.

Let $\calN$ be the normalizer of $\calR$ in $\Or(V_\ell)$.    Take any $g\in \calN$; conjugation by $g$ defines an automorphism of $\calR$.   Since $\#\operatorname{Out}(\calR)\leq 2$, there is an $r\in \calR$ such that $g^2 r^{-1}$ commutes with $\calR$.  Since $V_\ell$ is an absolutely irreducible $\calR$-module, $g^2 r^{-1}$ must be a scalar matrix.   Therefore, $g^2=\pm r$ since the only scalar matrices in $\Or(V_\ell)$ are $I$ and $-I$.   So $g^{2e}=I$ for some $e\in\{8, 10, 12, 14,18\}$.    We have $G\subseteq \calN$ since $\calR$ is a normal subgroup of $G$.   So for each $g\in G$, we have $g^{e} = I$ for some $e\in \{16,20,24,28,36\}$.  However, this contradicts Lemma~\ref{L:Lfunction orders}.   

Therefore, $\calR$ contains $\Omega(V_\ell)$.  We have $\calR \not\subseteq \SO(V_\ell)$ since reflections have determinant $-1$.   The lemma follows since $G^g\supseteq \calR$.
\\

%
%
%
%
%



\end{document}